\documentclass[11pt]{amsart}
\usepackage{amsmath, amsfonts, amsthm, amssymb}
\usepackage[all,cmtip,line]{xy}
\usepackage{array}
\usepackage{enumerate}
\usepackage{srcltx} % for inverse search in LINUX -- does not cause any problems in Windows

\newtheorem{thm}{Theorem}[section]
\newtheorem{lem}{Lemma}[section]
\newtheorem{prop}[lem]{Proposition}
\newtheorem{cor}[lem]{Corollary}
\newtheorem{defn}[lem]{Definition}
\newtheorem{rem}[lem]{Remark}
\newtheorem{conj}{Conjecture}[section]

\newtheorem{conj/q}{Conjecture/Open Question}[section]

\numberwithin{equation}{section}

\newcommand{\bR}{ \mathbb{R}} %%%% real
\newcommand \eps{\varepsilon}

\newlength{\originalbase}
\newcommand{\spacing}[1]{\setlength{\baselineskip}{#1\originalbase}}
\begin{document}               % PLUS THE \end COMMAND AT THE END.

\newcommand{\avint}{{- \hspace{-3.5mm} \int}}

\spacing{1}

%--------Meta Data: Fill in your info------
\title{ On the constant scalar curvature K\"ahler metrics (II)\\ ---existence results}
\author{Xiuxiong Chen, Jingrui Cheng}

\date{\today}

\maketitle
\begin{abstract}
In this paper, we generalize our apriori estimates on cscK(constant scalar curvature K\"ahler) metric equation  \cite{cc1} to more general scalar curvature type equations (e.g., twisted cscK metric equation). As applications, under the assumption that the automorphism group is discrete, we prove the celebrated Donaldson's conjecture that the non-existence of cscK metric is equivalent to the existence of a destabilized geodesic ray where the $K$-energy is non-increasing. Moreover, we prove that the properness of $K$-energy in terms of $L^1$ geodesic distance $d_1$ in the space of K\"ahler potentials implies the existence of cscK metric.  Finally, we prove that weak minimizers of the $K$-energy in $(\mathcal{E}^1, d_1)$ are smooth.
The continuity path proposed in \cite{chen15} is instrumental in this proof. 
\end{abstract}

\tableofcontents

\section{Introduction}

This is the second of a series of papers discussing constant scalar curvature K\"ahler metrics. In this paper, for simplicity, we will only consider the case $Aut_0(M, J) = 0.\;$ Here $Aut_0(M,J)$ denotes the identity component of the automorphism group  and $Aut_0(M,J)=0$ means the group is discrete.  Under this assumption, we prove Donaldson's conjecture (mentioned above in the abstract) as well as the existence part of properness conjecture in this paper.
Our main method is to adopt the continuity path introduced in \cite{chen15} and we need to prove that the set of parameter $t\in [0,1]$ the continuity path is both open (c.f. \cite{chen15} ) and closed under suitable geometric constraints.  The apriori estimates obtained in \cite{cc1} 
and their modifications (where the scalar curvature takes twisted form as in the twisted path introduced in \cite{chen15}) are the crucial technical ingredients needed in this paper.  In the sequel of this paper,  
we will prove a suitable generalization of both conjectures for general automorphism groups (i.e. no longer assume they are discrete).\\

We will begin with a brief review of history of this problem. In 1982 and 1985, E. Calabi published
two seminal papers \cite{calabi82} \cite{calabi85} on extremal K\"ahler metrics where he proved some fundamental theorems on extremal K\"ahler metrics. His initial vision is that there should be a unique canonical  metric in each
K\"ahler class.   Levine's example(c.f \cite{Levine}) however shows that there is a K\"ahler class in iterated blowup of $\mathbb{C}\mathbb{P}^2$ which admits no extremal K\"ahler metrics.  More examples and obstructions are found over the last few decades and huge efforts are devoted to formulate the right conditions (in particular the algebraic conditions) under which we can ``realize" Calabi's original dream in a suitable format.
The well known Yau-Tian-Donaldson conjecture is one of the important formulations now which states that on projective manifolds, the cscK metrics exist in a polarized K\"ahler class if and only if this class is $K$-stable. It is widely expected among experts that the stability condition needs to be strengthened to a stronger notion such as uniform stability or stability through filtrations, in order to imply the existence of cscK metrics. We will have more in-depth discussions on this issue in the next paper in this series.   \\

In a seminal paper \cite{Dona96}, S. K. Donaldson proposed a beautiful program in K\"ahler geometry, aiming in particular to attack Calabi's renowned problem of existence of cscK metrics. In this celebrated program,   Donaldson took the point of view that the space of K\"ahler metrics is formally a symmetric space of non-compact type and
the scalar curvature function is the moment map from the space of almost complex structure compatible with a fixed symplectic form to the Lie algebra of certain infinite dimensional sympletic structure group which is exactly the space of all real valued smooth functions in the manifold.
With this in mind, Calabi's problem of finding a cscK metric is reduced to finding a
zero of this moment map in the infinite dimensional space setting. From this beautiful new point of view, S. K.  Donaldson proposed a network of problems in K\"ahler geometry which have inspired many exciting developments over the last two decades, culminating in the recent resolution of Yau's stability conjecture on K\"ahler-Einstein metrics \cite{cds12-1} \cite{cds12-2} \cite{cds12-3}. \\

 Let $\mathcal{H}$ denote the space of K\"ahler potentials 
in a given K\"ahler class $(M, [\omega])$.  T. Mabuchi\cite{Mabuchi},  S. Semmes \cite{semmes} and S. K. Donaldson \cite{Dona96}  set up an $L^2$ metric in the space of K\"ahler potentials:
\[
    \|\delta \varphi\|^2_\varphi = \displaystyle \int_M\; (\delta \varphi)^2 \omega_\varphi^n,\qquad \forall \; \delta \varphi \in T_\varphi \mathcal{H}.
\]
Donaldson \cite{Dona96} conjectured that $\mathcal{H}$ is a genuine metric space with the pathwise distance defined by this $L^2$ inner product.  In \cite{chen991}, the first named author established the existence of $C^{1,1}$
geodesic segment between any two smooth K\"ahler potentials and proved this conjecture of S.K. Donaldson. He went on to prove (together with E. Calabi) that such a space
is necessarily non-positively curved in the sense of Alexandrov\cite{calabi-chen}.  More importantly,  S. K. Donaldson proposed the following
 conjecture to attack the existence problem:
\begin{conj} \cite{Dona96} Assume $Aut_0(M, J)=0 $. Then the following statements are equivalent:
\begin{enumerate}
\item There is no constant scalar curvature K\"ahler metric in $\mathcal{H}$;
\item  There is a potential $\varphi_0 \in \mathcal{H}_0$ and there exists a geodesic ray $\rho(t) (t \in [0,\infty))$ in $\mathcal{H}_0$, initiating from $\varphi_0$ such
that the $K$-energy is non-increasing;
\item   For any K\"ahler potential $\psi \in \mathcal{H}_0$, there exists a geodesic ray $\rho(t) (t \in [0,\infty))$ in $\mathcal{H}_0$, initiating from $\psi$ such
that the $K$-energy is non-increasing.
\end{enumerate}
\end{conj}

In the above, $\mathcal{H}_0=\mathcal{H}\cap\{\phi:I(\phi)=0\}$, where the functional $I$ is defined by (\ref{IJ}).
The reason we need to use $\mathcal{H}_0$ is to preclude the trivial geodesic $\rho(t)=\varphi_0+ct$ where $c$ is a constant.

     In the original writing of S. K. Donaldson, he didn't specify the regularity of these geodesic rays in this conjecture. In this paper, we avoid this issue by working in the space $\mathcal{E}^1$ in which the potentials have only very weak regularity but the notion of geodesic still makes sense. Moreover, Theorem 4.7 of \cite{Darvas1602} shows the definition of $K$-energy can be extended to the space $\mathcal{E}^1$.  The precise version of the result we prove is the following:
        
        \begin{thm}(Theorem \ref{t3.1})\label{t1.1new} Assume $Aut_0(M, J)=0 $.  Then the following statements are equivalent:
\begin{enumerate}
\item There is no constant scalar curvature K\"ahler metric in $\mathcal{H}$;
\item  There is a potential $\varphi_0 \in \mathcal{E}^1_0$ and there exists a locally finite energy geodesic ray $\rho(t) (t \in [0,\infty))$ in $\mathcal{E}_0^1$, initiating from $\varphi_0$ such
that the $K$-energy is non increasing;
\item   For any K\"ahler potential $\psi \in \mathcal{E}^1_0$, there exists a locally finite energy geodesic ray $\rho(t) (t \in [0,\infty))$ in $\mathcal{E}_0^1$, initiating from $\psi$ such
that the $K$-energy is non increasing.
\end{enumerate}
        \end{thm}
In the above, the space $ \mathcal{E}^1$ is the abstract metric completion of the space $\mathcal{H}$ under the Finsler metric $d_1$ in $\mathcal H$ (see section 2 for more details) and the notion of finite energy geodesic segment was introduced in \cite{BBEGZ} (c.f.  \cite{Darvas1402}).
Also $\mathcal{E}_0^1=\mathcal{E}^1\cap\{\phi:I(\phi)=0\}$, where the functional $I$ is defined as in (\ref{IJ}). The idea of using locally finite energy geodesic ray is inspired by the recent beautiful work
of Darvas-He \cite{Darvas-He} on Donaldson conjecture in Fano manifold where they use Ding functional instead of the $K$-energy functional.  From our point of view, both the restriction to canonical
K\"ahler class and the adoption of Ding functional are more of analytical nature.  \\

Inspired by Donaldson's conjecture, the first named author introduced the following notion of geodesic stability \cite{chen05}.
\begin{defn}\label{t1.1n} (c.f. Definition (3.10) in \cite{chen05}) Let $\rho(t): [0,\infty) \rightarrow {\mathcal {E}}_0^1$ be a locally finite energy geodesic ray with unit speed. One can define an invariant $\yen([\rho])$ as
\[
\yen([\rho]) = \displaystyle \lim_{k\rightarrow \infty}\; K(\rho(k+1))-K(\rho(k)).
\]
\end{defn}
One can check that this is well defined, due to the convexity of $K$-energy along geodesics (c.f. Theorem \ref{t2.4new}). Indeed, from the convexity of $K$-energy along locally finite energy geodesic ray, one actually has $K(\rho(k+1))-K(\rho(k))$ is increasing in $k$.
\begin{defn}\label{d1.2n}(c.f. Definition (3.14) in \cite{chen05}) \sloppy Let $\varphi_0\in\mathcal{E}_0^1$ with $K(\varphi_0)<\infty$, $( M,[\omega])$ is called geodesic stable at $\varphi_0$(resp. geodesic-semistable) if for all locally finite energy geodesic ray initiating from $\varphi_0$, their $\yen$ invariant is always strictly positive(resp. nonnegative).
$(M,[\omega])$ is called geodesic stable(resp. geodesic semistable) if it is geodesic stable(resp. geodesic semistable) at any $\varphi\in\mathcal{E}_0^1$. \end{defn}

\begin{rem} It is possible to define the $\yen$ invariant for a locally finite energy geodesic ray in $\mathcal{E}_0^p$ with $p>1$.   Note that a geodesic segment in $\mathcal{E}_0^p$ is automatically a geodesic segment in $\mathcal{E}_0^q$ for any $q \in [1,p].\;$ Following the preceding definition, one can also define geodesic stability in $\mathcal{E}_0^p (p>1).\;$  Note that for a locally given finite energy geodesic ray in   $\mathcal{E}_0^p (p>1), $ the actual value of $\yen$ invariant  in $\mathcal{E}_0^p$ might
differ by a positive multiple from the $\yen$ invariant considered in $\mathcal{E}_0^1$. However, it will not affect the sign of the $\yen$ invariant for a particular locally finite energy geodesic ray. On the other hand, the collection of locally finite energy
geodesic ray in  $\mathcal{E}_0^p (p>1)$ might be strictly contained in the collection of geodesic rays in $\mathcal{E}_0^1.\;$  Therefore, the notion of geodesic stability in the $\mathcal{E}_0^1$ is strongest
while the notion of geodesic stability in $\mathcal{E}_0^\infty$ is the weakest.  Without going into technicality, we may define geodesic stability in $\mathcal{E}_0^\infty$ as the $\yen$ invariant being strictly positive for any locally finite energy geodesic ray which lies in $\displaystyle \bigcap_{p\geq 1}  \mathcal{E}_0^p.\;$  \\
\end{rem}

An intriguing question motivated from above remark is whether geodesic stability in $\mathcal{E}_0^{\infty}$ (in the sense defined in the above remark) implies geodesic stability in $\mathcal{E}_0^1$?
The first named author believes the answer is affirmative. We will discuss this question and other stability notions in algebraic manifolds in greater detail in our next paper and refer interested readers to the following works and references therein:
%\begin{conj/q} In a polarized algebraic manifold $(M, L, [\omega])$,  the following notions of stability are equivalent.
%\begin{enumerate}
%\item Stable in the sense of filtrations;
%\item Uniform Stable;
%\item Geodesic stable in $\mathcal{E}_0^\infty$; 
%\item Geodesic stable in $\mathcal{E}_0^p (1<p < \infty);$ 
%\item  Geodesic stable in $\mathcal{E}_0^1. $ 
%\end{enumerate}
%\end{conj/q}
%Here uniform stability and stability through filtrations are algebraic notions and we refer to the work of
J. Ross \cite{Ross05}, G. Sz$\acute{\text{e}}$kelyhidi \cite{Sz06},  Berman-Boucksom-Jonsson \cite{BBJ15},  R. Dervan \cite{Dervan171}.\\

  Using the notion of geodesic stability, we can re-formulate Theorem \ref{t1.1n} as
\begin{thm}\label{t1.2n} Suppose $Aut_0(M, J)=0$. Then  $(M,[\omega])$ admits a cscK metric if and only if it is geodesic stable.
\end{thm}
      Given the central importance of the notion of $K$-energy in Donaldson's beautiful program, the first named author proposed the following conjecture,  shortly after \cite{chen991}:
      \begin{conj}\label{conj1.2} Assume $Aut_0(M,J)=0$. The existence of constant scalar curvature K\"ahler metric is equivalent to the properness of $K$-energy in terms of geodesic
      distance.
      \end{conj}
   Here ``properness" means that the $K$-energy tends to $+\infty$ whenever the geodesic distance tends to infinity (c.f. Definition \ref{d4.1}). The original conjecture naturally chose the distance introduced in \cite{Dona96}
      which we now call $L^2$ distance.  After a series of fundamental work of T. Darvas on this subject (c.f \cite{Darvas1403}  \cite{Darvas1402}), we now learn  that the $L^1$ geodesic distance is a natural choice for the properness conjecture. Indeed, we prove
      \begin{thm}(Theorem \ref{t2.2} and \ref{t4.2n}) \label{t1.2nn} Assume $Aut_0(M,J)=0$. The existence of constant scalar curvature K\"ahler metric is equivalent to the properness of $K$-energy in terms of the $L^1$
      geodesic distance.
      \end{thm}
      
      Note that the direction that existence of cscK implies properness has been established by Berman-Darvas-Lu\cite{Darvas1605} recently. For the converse (namely the existence part), Darvas and Rubinstein have reduced this problem in \cite{DR} to a question of regularity of minimizers.
In our paper, we will use continuity method to bypass this question and establish existence of cscK metrics.\\

  For properness conjecture, we remark that there is a more well known formulation due to G. Tian where he conjectured that the existence
      of cscK metrics is equivalent to the propeness of $K$-energy in terms of Aubin functional $J$ (c.f. Definition (\ref{IJ})).  One may say that Tian's conjecture is more of analytical nature while Conjecture 1.2 above fits into Donaldson's geometry program in the space of K\"ahler potentials
      more naturally.  According to T. Darvas (c.f. Theorem 5.5 of \cite{Darvas1403}),  Aubin's $J$ functional and the $L^1$ distance are equivalent. Therefore, these two
      properness conjectures are equivalent.  Nonetheless, the formulation in conjecture 1.2  is essential to our proof.\\

%We also remark that Conjecture \ref{conj1.2} can also be formulated when $Aut_0(M,J)$ is nontrivial, as long as one replaces the $L^1$ geodesic distance by the $L^1$ distance under group action.
%This will be discussed in more detail in the next paper in the sequel.\\

Theorem \ref{t1.2nn} also holds for twisted cscK metric as well (c.f. Theorem \ref{t2.2} \ref{t4.2n}), which is the solution to the equation 
$$
t(R_{\varphi}-\underline{R})=(1-t)(tr_{\varphi}\chi-\underline{\chi}).$$ In the above, $0<t\leq 1$, $\chi$ is a fixed K\"ahler form, and $\underline{R}$, $\underline{\chi}$ are suitable constants determined by the K\"ahler classes $[\omega_0]$, $[\chi]$.

Now we recall an important notion introduced in \cite{chen15}:
\begin{equation}\label{1.1nn}
R([\omega_0],\chi)=\sup\{t_0\in[0,1]:\textrm{ the above equation can be solved for any $0\leq t\leq t_0$.}\}
\end{equation}
In the same paper, the first named author conjectured that this is an invariant of the K\"ahler class $[\chi]$. In this paper, 
as a consequence of Theorem \ref{t2.2} and \ref{t4.2n},  we will show that if $\chi_1$ and $\chi_2$ are two K\"ahler forms in the same class, then one has
 $$R([\omega_0],\chi_1)=R([\omega_0],\chi_2), $$ so that the quantity $R([\omega_0],[\chi])$ is well-defined and gives rise to an invariant between two K\"ahler classes $[\omega_0],\,[\chi].\;$ Moreover, when the $K$-energy is bounded from below, the twisted path (\ref{2.12}) can be solved for any $t<1$, as long as $t=0$ can be solved. Thus in this case we have

\begin{thm} Let $\chi$ be a K\"ahler form. If the $K$-energy is bounded from below on $(M,[\omega])$, then $R([\omega_0],[\chi]) = 1$ if and only if $R([\omega_0],[\chi]) >0. $
\end{thm}

As noted in \cite{chen15}, it is interesting to understand geometrically for what K\"ahler classes this invariant is $1$ but do not admit constant scalar curvature metrics. More broadly,
it is interesting to estimate the upper and lower bound of this invariant.  It is not hard to see the relation between the invariant introduced in \cite{Sz11} and the invariant introduced above when restricted to the canonical K\"ahler class in Fano manifold, where we take $[\chi]$ to be the first Chern class in (\ref{1.1nn}) above. 
Hopefully, the method used there can be adapted to our setting to get estimate for this new invariant, in particular an upper bound. %After all, one estimates upper bound in \cite{Sz11}
%via estimates when the configuration is not slope stable. 
% It is plausible that the same kind of calculation may extend over to cscK case.\\

 T. Darvas and  Y. Rubinstein conjectured in \cite{DR}(Conjecture 2.9) that any minimizer of $K$-energy over the space $\mathcal{E}^1$ is actually a smooth K\"ahler potential. This 
is a bold and  imaginative  conjecture which might  be viewed as a natural generalization of an earlier conjecture by the first named author that any $C^{1,1}$ minimizer of $K$-energy is smooth (c.f. \cite{chen00}, Conjecture 3). Under an additional assumption that there exists a smooth cscK metric in the same K\"ahler class, 
Darvas-Rubinstein conjecture is verified in  \cite{Darvas1605}. 
In this paper, we establish this conjecture as an application of properness theorem. Note that Euler-Lagrange equation is not available apriori in our setting, so that the usual approach to the regularity problem in the calculus of variations does not immediately apply. Instead, we need to use the continuity path to overcome this difficulty. 

\begin{thm}\label{t1.3}(Theorem \ref{t4.1}) Let $\varphi_*\in\mathcal{E}^1$ be such that $K(\varphi_*)=\inf_{\varphi\in\mathcal{E}^1}K(\varphi)$.
Then $\varphi_*$ is smooth and $\omega_{\varphi_*}:=\omega_0+\sqrt{-1}\partial\bar{\partial}\varphi_*$ is a cscK metric.
\end{thm}
We actually establish a more general result which allows us to consider more general twisted $K$-energy and we can show the weak minimizers of twisted $K$-energy are smooth as long as the twisting form is smooth, closed and nonnegative.

   \begin{rem}W. He and Y. Zeng \cite{He-Zeng}  proved Chen's conjecture on the regularity of $C^{1,1}$ minimizers of $K$-energy. Their original proof contains an unnecessary assumption that the $(1,1)$ current defined by the minimizer has a strictly positive lower bound
   which can be removed by adopting a weak K\"ahler-Ricci flow method initiated in Section 7 of  Chen-Tian \cite{CT}.  This will be discussed in an unpublished note \cite{cc0}.
\end{rem}
%In \cite{Darvas1602}, the author showed that the definition of $K$-energy can be extended to the space $\mathcal{E}^1$.  

 In view of Theorem 1.3, it is important to study, under what conditions, the $K$-energy functional is proper in a given K\"ahler class. In \cite{chen00},   the first named author proposed a decomposition formula for $K$-energy:
       \begin{equation}\label{1.2nn}
K(\varphi)=\int_M\log\bigg(\frac{\omega_{\varphi}^n}{\omega_0^n}\bigg)\frac{\omega_{\varphi}^n}{n!}+J_{-Ric}(\varphi).
\end{equation}
where the functional $J_{-Ric}$ is defined through its derivatives:
\begin{equation}\label{1.2new}
{{d\, J_{-Ric}}\over {d\,t}} = \displaystyle \int_M\; {{\partial \varphi}\over {\partial t}} (-Ric\wedge\frac{\omega_{\varphi}^{n-1}}{(n-1)!}+\underline{R}\frac{\omega_{\varphi}^n}{n!}).
\end{equation}
 One key observation in \cite{chen00} (based on this decomposition formula) is that $K$-energy has a lower bound if the corresponding $J_{-Ric}$ functional has a lower bound.  Note that when the first Chern class
is negative, one can choose a background metric such that $- Ric > 0.\;$ Then, $J_{-Ric}$ is  convex along $C^{1,1}$ geodesics in $\mathcal H$ and is bounded from below if it has a critical point. In \cite{SongBen040},  Song-Weinkove further pointed out that, $J_{-Ric}$ functional being bounded from below is sufficient to imply the properness of $K$-energy.
 The research in
this direction has been very active and intense (c.f. Chen\cite{chen00}, Fang-Lai-Song-Weinkove \cite{FLSW14}, Song-Weikove \cite{SongBen13},  Li-Shi-Yao \cite{LSY13},  R. Dervan \cite{Dervan142}, and references therein).
%In stead of $J_{Ric}$, one may also consider $J_{\chi}$, given by (\ref{1.2new}) after replacing $-Ric$ by any K\"ahler form $\chi$. The study of critical points of $J_{\chi}$ functional
%s thus active in itself over last decades and 
Combining these results with Theorem \ref{t1.2nn}, we have the following corollary.

\begin{cor}\label{c1.5n} There exists a cscK metric in $(M,[\omega])$ if any one of the following conditions holds:
\begin{enumerate}

\item There exists  a constant $\epsilon \geq 0$ such that $ \epsilon < {{n+1}\over n} \alpha_M([\omega]) $ and $ \pi C_1(M) < \epsilon [\omega]$
such that
\[
\left( - n {{C_1(M) \cdot [\omega]^{n-1}}\over {[\omega]^n} } +\epsilon\right) \cdot [\omega] + (n-1) C_1(M) > 0.
\]
Here $\alpha_M(\omega)$ denotes the $\alpha$-invariant of the K\"ahler class $(M,[\omega])$ (c.f. \cite{tian87}).
\item  If 
\[
\alpha_M([\omega]) >  {{C_1(M) \cdot [\omega]^{n-1}}\over {[\omega]^n} } \cdot {n\over {n+1}}
\]
and 
\[
C_1(M) \geq  {{C_1(M) \cdot [\omega]^{n-1}}\over {[\omega]^n} } \cdot {n\over {n+1}} \cdot [\omega].
\]
\end{enumerate}
\end{cor} 

Here part (i) of Corollary \ref{c1.5n} follows Theorem \ref{t1.2nn} and  Li-Shi-Yao \cite{LSY13} (c.f. Fang-Lai-Song-Weinkove \cite{FLSW14} Song-Weinkove \cite{SongBen13}), part (ii) of Corollary \ref{c1.5n} follows Theorem \ref{t1.2nn}
and R. Dervan \cite{Dervan142}.  
%\begin{rem} Condition in item 2 imply it is slope stable by work of Panov-Ross.
%\end{rem}

Following Donaldson's observation in \cite{Dona99}, if a K\"ahler surface $M$ admits no curve of negative self intersections and has $C_1(M)<0$, then the condition 
\[
   {{2 [\omega]\cdot [-C_1(M)]}\over [\omega]^2} \cdot [\omega]  -[-C_1(M)]> 0
\]
is satisfied automatically for any K\"ahler class $[\omega]$ (c.f.  Song-Weinkove \cite{SongBen040}). Consequently, on 
any K\"ahler surface $M$ with $C_1(M) < 0$  with no curve of negative self-intersection,    the $K$-energy is proper for any K\"ahler class (c.f. Song-Weinkove \cite{SongBen13}). It follows that on these surfaces, every K\"ahler class admits a cscK metric.  Albeit restrictive, this is indeed very close to the original vision of E. Calabi that every
K\"ahler class should have one canonical representative.
E. Calabi's vision has inspired generations of K\"ahler geometers to
work on this exciting problem and without it, this very paper will never exist.  To celebrate his vision, we propose to call such a manifold 
a {\it Calabi dream manifold}.
\begin{defn}  A K\"ahler manifold is called {\bf Calabi dream manifold} if every K\"ahler class on it admits an extremal K\"ahler metric.
\end{defn}

Clearly, all compact Riemann surfaces, complex projective spaces $\mathbb{C}\mathbb{P}^n$ and all compact Calabi-Yau manifolds \cite{Yau78} are {\it Calabi dream manifolds}. Our discussion above asserts
\begin{cor} Any K\"ahler surface with $C_1 < 0$ and no curve of negative self-intersection is a Calabi dream surface. 
\end{cor}

 It is fascinating to understand how large this family of Calabi dream surfaces is.  We will delay more discussions on Calabi dream manifolds to the end of Section 2.\\

The key technical theorem we prove is the following compactness theorem in the space of K\"ahler potentials:
\begin{thm}(Corollary \ref{c1.1})\label{t1.6}  The set of K\"ahler potentials(suitably normalized up to a constant) with bounded scalar curvature and entropy (or geodesic distance) is bounded in $W^{4,p}$ for any  $p<\infty$, hence precompact in $C^{3,\alpha}$ for any $0<\alpha<1$.
\end{thm}
This is an improvement from earlier work of first named author, Theorem 1.4 \cite{chen05},  where he additionally assumed a bound on Ricci curvature. More recently,  Chen-Darvas-He \cite{chd} proved that the set of K\"ahler potentials with uniform Ricci upper bound and $L^1$ geodesic distance bound is precompact in $C^{1,\alpha} $ for any $0<\alpha<1$ (indeed, the K\"ahler form is bounded from above).   As a corollary of Theorem 1.6, we prove \begin{thm}\label{t1.7new} The Calabi flow can be extended as long as the scalar curvature is uniformly bounded.
\end{thm}
\begin{rem} This is a surprising development.  With completely different motivations in geometry, the first named author has a similar conjecture on Ricci flow which states that the only obstruction to the long time existence of 
Ricci flow is the $L^\infty$ bound of scalar curvature.  There has been significant progress in this problem, first by a series of works of B. Wang (c.f. \cite{wang12}, \cite{cw})  and more recently 
by the interesting and important work of Balmer-Zhang \cite{BZ17} and M. Simons \cite{MS15} in dimension 4.
\end{rem}
Theorem \ref{t1.7new} is a direct consequence of Theorem \ref{t1.6} and Chen-He short time existence theorem (c.f. Theorem 3.2 in  \cite{chenhe05}), where the authors proved the life span of the short time  solution depends only on $C^{3,\alpha}$ norm of the initial K\"ahler potential and lower bound of the initial metric. By assumption, we know that $\partial_t\varphi$ remains uniformly bounded, hence $\varphi$ is bounded on every finite time interval.
On the other hand, since $K$-energy is decreasing along the flow, in particular $K$-energy is bounded from above along the flow.
Due to (\ref{1.2nn}) and that $\varphi$ is bounded, we see that the entropy is bounded as well.
Hence the flow remains in a precompact subset of $C^{3,\alpha}(M)$ on every finite time interval, hence can be extended.\\

In light of Theorem 1.7 and a compactness theorem  of  Chen-Darvas-He \cite{chd}, a natural question is if one can extend the Calabi flow assuming only an upper bound  on Ricci curvature. A more difficult question is whether one-sided bound of the scalar curvature is sufficient for the extension of Calabi flow. Ultimately,  the remaining fundamental question is 
\begin{conj} (Calabi, Chen) Initiating from any smooth K\"ahler potential, the Calabi flow always exists globally.
\end{conj}
Given the recent work by J. Street\cite{Street12}, Berman-Darvas-Lu\cite{Darvas1602}, the weak Calabi flow always exists globally. Perhaps one can prove this conjecture via improving regularity of weak Calabi flow. On the other hand, one may hope to prove this conjecture on K\"ahler classes which already admit constant scalar curvature K\"ahler metrics and prove the flow
will converges to such a metric as $t\rightarrow\infty$. An important and deep result in this direction is Li-Wang-Zheng's work \cite{LWZ15}.\\

Finally  we explain the organization of the paper:

In section 2, we recall the necessary preliminaries needed for our proof, including the continuity path we will use to solve the cscK equation and the theory of geodesic metric spaces established by Darvas and others.

In section 3, we generalize our previous estimates in \cite{cc1} on cscK equation to more general type of equations, so that we can apply these estimates to twisted cscK equation and Calabi flow.

In section 4, we prove the equivalence between the existence of cscK metric and properness of $K$-energy, namely Theorem \ref{t1.2nn}.

In section 5, we prove that a minimizer of $K$-energy over the space $\mathcal{E}^1$ is smooth. More general twisted $K$-energy is also considered and we show its minimizer is smooth as long as the twisting form is nonnegative, closed and smooth.

In section 6, we show that the existence of cscK metric is equivalent to geodesic stability, In particular, we verify the Donaldson's conjecture, Theorem \ref{t1.1new}.\\

\noindent{\bf Acknowledgement.}  Both authors are grateful to the help from the first named author's colleague Professor Jason Starr in the discussions about {\it Calabi dream manifolds}.

\section{preliminaries}
In this section, we will review some basic concepts in K\"ahler geometry as well as some fundamental results involving finite energy currents,
which will be needed for our proof of Theorem 1.1 and 1.3. In particular, it includes the characterization of the space $(\mathcal{E}^1,d_1)$, a compactness result on bounded subsets of $\mathcal{E}^1$ with finite entropy. We also include results on the convexity of $K$-energy along $C^{1,1}$ geodesics as well as its extension to the space $\mathcal{E}^1$. 
For more detailed account on these topics, we refer to a recent survey paper by Demailly \cite{demailly}.  At the end of this section, we will discuss about Calabi dream manifolds.

\subsection{$K$-energy and twisted $K$-energy}
Let $(M,\omega_0)$ be a fixed K\"ahler class on $M$.
Then we can define the space $\mathcal{H}$ of K\"ahler metrics cohomologous to $\omega_0$ as:
\begin{equation}
\mathcal{H}=\{\varphi\in C^2(M):\omega_{\varphi}:=\omega_0+\sqrt{-1}\partial\bar{\partial}\varphi>0\}.
\end{equation}
We can introduce the $K$-energy in terms of its derivative:
\begin{equation}
\frac{dK}{dt}(\varphi)= - \int_M {{\partial \varphi} \over {\partial t}} (R_{\varphi} - \underline{R})\frac{\omega_{\varphi}^n}{n!},\textrm{ $\varphi\in\mathcal{H}$.}
\end{equation}
Here $R_{\varphi}$ is the scalar curvature of $\omega_{\varphi}$, and 
$$\underline{R} = {{[C_1(M)] \cdot [\omega]^{[n-1]}} \over {[\omega]^{[n]}}} = {{\int_M R_\varphi \omega_\varphi^n}\over {\int_M \omega^n}}.\;$$ 
Following \cite{chen00}, we can write down an explicit formula for $K(\varphi)$:
\begin{equation}\label{K}
K(\varphi)=\int_M\log\bigg(\frac{\omega_{\varphi}^n}{\omega_0^n}\bigg)\frac{\omega_{\varphi^n}}{n!}+J_{-Ric}(\varphi),
\end{equation}
where for a $(1,1)$ form $\chi$, we define
\begin{equation}\label{J-chi}
\begin{split}
J_{\chi}&(\varphi)=\int_0^1\int_M\varphi\bigg(\chi\wedge\frac{\omega_{\lambda\varphi}^{n-1}}{(n-1)!}-\underline{\chi}\frac{\omega_{\lambda\varphi}^n}{n!}\bigg)d\lambda\\
&\qquad \qquad =\frac{1}{n!}\int_M\varphi\sum_{k=0}^{n-1}\chi\wedge\omega_0^k\wedge\omega_{\varphi}^{n-1-k}-\frac{1}{(n+1)!}\int_M\underline{\chi}\varphi\sum_{k=0}^n\omega_0^k\wedge\omega_{\varphi}^{n-k}.
\end{split}
\end{equation}
Here $$\underline{\chi}=\frac{\int_M\chi\wedge\frac{\omega_0^{n-1}}{(n-1)!}}{\int_M\frac{\omega_0^n}{n!}}.$$ Following formula (\ref{1.2new}), we have 
 $$\frac{dJ_{\chi}}{dt}=\int_M\partial_t\varphi(tr_{\varphi}\chi-\underline{\chi})\frac{\omega_{\varphi}^n}{n!}.$$
It is well-known that $K$-energy is convex along smooth geodesics in the space of K\"ahler potentials. \\

Let $\beta\geq0$ be a smooth closed $(1,1)$ form, we define a ``{\it twisted  $K$-energy with respect to $\beta$}" by
\begin{equation}\label{K-beta}
K_{\beta}(\varphi)=K(\varphi)+J_{\beta}(\varphi).
\end{equation}

The critical points of $K_{\beta}(\varphi)$ satisfy the following equations:

\begin{equation}\label{2.6n}
R_{\varphi}-\underline{R}=tr_{\varphi}\beta-\underline{\beta}, \;\;\;{\rm where}\;\; \underline{\beta}=\frac{\int_M\beta\wedge\frac{\omega_0^{n-1}}{(n-1)!}}{\int_M\frac{\omega_0^n}{n!}}.
\end{equation}
%In the above, $\underline{\beta}=\frac{\int_M\beta\wedge\frac{\omega_0^{n-1}}{(n-1)!}}{\int_M\frac{\omega_0^n}{n!}}$.
For later use, we also define the functionals $I(\varphi),  J(\varphi) $, given by
\begin{equation}\label{IJ}
I(\varphi)=\frac{1}{(n+1)!}\int_M\varphi\sum_{k=0}^n\omega_0^k\wedge\omega_{\varphi}^{n-k},  \qquad J(\varphi)=\int_M\varphi(\omega_0^n-\omega_{\varphi}^n).
\end{equation}

We also need to consider the more general twisted $K$-energy, 
which is defined to be
\begin{equation}\label{2.10}
K_{\chi,t}=t K+(1-t)J_{\chi}.
\end{equation}
Following \cite{chen00}, we can write down Euler-Lagrange equation for twisted $K$-energy:
\begin{equation}\label{2.12}
t(R_{\varphi}-\underline{R})=(1-t)(tr_{\varphi}\chi-\underline{\chi}),\,\,t\in[0,1].
\end{equation}
Following \cite{chen15}, for $t>0$, we can rewrite this into two coupled equations:
\begin{align}
\label{twisted-1}
&\det(g_{i\bar{j}}+\varphi_{i\bar{j}})=e^F\det g_{i\bar{j}},\\
\label{twisted-2}
&\Delta_{\varphi}F=-(\underline{R}-\frac{1-t}{t}\underline{\chi})+tr_{\varphi}(Ric-\frac{1-t}{t}\chi).
\end{align}
In the following, we will assume $\chi>0$, that is, $\chi$ is a K\"ahler form. 
The equation
(\ref{2.12}) with $t \in [0,1]$ is the continuity path proposed in \cite{chen15} to solve the cscK equation.
More generally, one can consider similar twisted paths in order to solve (\ref{2.6n}). Namely we consider
\begin{equation}\label{2.13nn}
t(R_{\varphi}-\underline{R})=t(tr_{\varphi}\beta-\underline{\beta})+(1-t)(tr_{\varphi}\chi-\underline{\chi}).
\end{equation}
The solution to (\ref{2.13nn}) is a critical point of $tK_{\beta}+(1-t)J_{\chi}$. We will see later that it is actually a minimizer. 
For $t>0$, this again can be equivalently put as
\begin{align}
\label{g-twisted1}
&\det(g_{i\bar{j}}+\varphi_{i\bar{j}})=e^F\det g_{i\bar{j}},\\
\label{g-twisted2}
&\Delta_{\varphi}F=-(\underline{R}-\underline{\beta}-\frac{1-t}{t}\underline{\chi})+tr_{\varphi}\big(Ric-\beta-\frac{1-t}{t}\chi\big).
\end{align}
An important question  is whether the set of $t$ for which (\ref{2.13nn}) can be solved is open.
The cited result is only for (\ref{2.12}), but the same argument would work for (\ref{2.13nn}).
\begin{lem}\label{l2.2}  (\cite{chen15}, \cite{zeng}, \cite{Hashi}):
Suppose for some $0\leq t_0<1$, (\ref{2.13nn}) 
has a solution $\varphi\in C^{4,\alpha}(M)$ with $t=t_0$, then for some $\delta>0$, (\ref{2.13nn}) has a solution in $C^{4,\alpha}$ for any $t\in (t_0-\delta,t_0+\delta) \bigcap [0, 1)$.
\end{lem}
We observe that we can always make sure (\ref{2.12}) or (\ref{2.13nn}) can be solved for $t=0$ by choosing $\chi=\omega_0$ or any K\"ahler form in $[\omega_0].\;$
\begin{rem}\label{r2.3}
Clearly if $\chi$ is smooth, it is easy to see by bootstrap that a $C^{4,\alpha}$ solution to (\ref{2.12}) is actually smooth.
Hence Lemma \ref{l2.2} shows the set of $t$ for which (\ref{2.12}) has a smooth solution is relatively open in $[0,1)$.
\end{rem}

\subsection{The complete geodesic metric space $(\mathcal{E}^p,d_p)$}    
Following Mabuchi, T. Darvas  \cite{Darvas1402} introduced the notion of  $d_1$ on $\mathcal{H}$.
\begin{equation}
||\xi||_{\varphi}=\int_M|\xi|\frac{\omega_{\varphi}^n}{n!},\forall \;\textrm{  $\xi\in T_{\varphi}\mathcal{H}=C^{\infty}(M)$.}
\end{equation}
Using this, we can define the path-length distance $d_1$ on the space $\mathcal{H}$, i.e. $d_1(u_0,u_1)$ equals the infimum of length of all smooth curves in $\mathcal H$,
with $\alpha(0)=u_0$, $\alpha(1)=u_1$.  Following  Chen \cite{chen991}, T. Darvas proved (\cite{Darvas1402}, Theorem 2) that
$(\mathcal{H},d_1)$ is a metric space. \\

In section 3.3 of \cite{GZ07} introduced the following space for any $p\geq 1$:
\begin{equation}
\mathcal{E}^p=\{\varphi\in PSH(M,\omega_0):\int_M\omega_{\varphi}^n=\int_M\omega_0^n,\,\,\int_M|\varphi|^p\omega_{\varphi}^n<\infty\}.
\end{equation}
 A fundamental conjecture of  V. Guedj \cite{Guedj14} stated that the completion of the space 
${\mathcal H}$  of smooth potentials equipped with the $L^2$ metric is precisely the space $ {\mathcal E}^2 (M, \omega_0)$ of potentials of finite  energy.  
This has been shown by Darvas  \cite{Darvas1402}, \cite{Darvas1403}, in which he has shown similar characterization holds for general $L^p$ metric.   Note that the extension to the $L^1$ metric is essential and fundamental to our work.  We have the following characterization for $(\mathcal{E}^1,d_1)$:
\begin{thm}\label{t2.3}
(\cite{Darvas1402}, Theorem 5.5)Define
$$
I_1(u,v)=\int_M|u-v|\frac{\omega_{u}^n}{n!}+\int_M|u-v|\frac{\omega_{v}^n}{n!},\textrm{ $u,v\in\mathcal{H}$.}
$$
Then there exists a constant $C>0$ depending only on $n$, such that 
\begin{equation}
\frac{1}{C}I_1(u,v)\leq d_1(u,v)\leq C I_1(u,v),\textrm{ for any $u,v\in\mathcal{H}$.}
\end{equation}
\end{thm}
For later use, here we describe how to obtain ``finite energy geodesics" from the $C^{1,1}$ geodesics between smooth potentials.

\begin{thm}\label{t2.2new}
(\cite{Darvas1402}, Theorem 2)
The metric completion of $(\mathcal{H},d_1)$ equals $(\mathcal{E}^1,d_1)$ where 
$$
d_1(u_0,u_1)=:\lim_{k\rightarrow\infty}d_1(u_0^k,u_1^k),
$$
for any smooth decreasing sequence $\{u_i^k\}_{k\geq1}\subset\mathcal{H}$ converging pointwise to $u_i\in\mathcal{E}^1$.
Moreover, for each $t\in(0,1)$, define
$$
u_t:=\lim_{k\rightarrow\infty}u_t^k,\,\,t\in(0,1),
$$
\sloppy  where $u_t^k$ is the $C^{1,1}$ geodesic connecting $u_0^k$ and $u_1^k$ (c.f.  \cite{chen991}).
We have $u_t\in\mathcal{E}^1$, the curve $[0,1]\ni t\mapsto u_t$ is independent of the choice of approximating sequences and is a $d_1$-geodesic in the sense that for some $c>0$, $d_1(u_t,u_s)=c|t-s|$, for any $s,\,t\in[0,1]$.
\end{thm}
The above limit is pointwise decreasing limit.
Since the sequence $\{u_i^k\}_{k\geq1}$ is decreasing sequence for $i=0$, $1$, we know $\{u_t^k\}_{k\geq1}$ is also decreasing for $t\in(0,1)$, by comparison principle.

We say $u_t:[0,1]\ni t\rightarrow\mathcal{E}^1$ connecting $u_0$, $u_1$ is a finite energy geodesic if it is given by the procedure described in Theorem \ref{t2.2new}.
The following result shows the limit of finite energy geodesics is again a finite energy geodesic.
\begin{prop}\label{p2.4new}(\cite{Darvas1602} , Proposition 4.3)
Suppose $[0,1]\ni t\rightarrow u_t^i\in\mathcal{E}^1$
 is a sequence of finite energy geodesic segments 
such that $d_1(u_0^i,u_0),\,d_1(u_1^i,u_1)\rightarrow0$. Then $d_1(u_t^i,u_t)\rightarrow0$, for any $t\in[0,1]$, where $[0,1]\ni t\mapsto u_t\in\mathcal{E}^1$ is the finite energy geodesic connecting $u_0$, $u_1$.
\end{prop}

Finally we record the following compactness result which will be useful later. This result was first established in \cite{BBEGZ}.
The following version is taken from \cite{Darvas1602}, which is the form most convenient to us.
\begin{lem}\label{l2.6new}(\cite{BBEGZ}, Theorem 2.17, \cite{Darvas1602}, Corollary 4.8)
Let $\{u_i\}_i\subset\mathcal{E}^1$ be a sequence for which the following condition holds:
$$
\sup_id_1(0,u_i)<\infty,\,\,\sup_iK(u_i)<\infty.
$$
Then $\{u_i\}_i$ contains a $d_1$-convergent subsequence.
\end{lem}

\subsection{Convexity of $K$-energy}
In this subsection, we record some known results about the convexity of $K$-energy and $J_{\chi}$ functional along $C^{1,1}$ geodesics and also finite energy geodesics.
In \cite{chen00}, the first named author proved the following result about the convexity of the functional $J_{\chi}$.
\begin{thm}(\cite{chen00}, Proposition 2)
Let $\chi\geq0$ be a closed $(1,1)$ form. Let $u_0$, $u_1\in\mathcal{H}$. 
Let $\{u_t\}_{t\in[0,1]}$ be the $C^{1,1}$ geodesic connecting $u_0$, $u_1$. 
Then $[0,1]\ni t\mapsto J_{\chi}(u_t)$ is convex.
\end{thm}
The convexity of $K$-energy is more challenging and the first named author made the following conjecture:
\begin{conj}\label{conj2.1}(Chen)
Let $u_0$, $u_1\in\mathcal{H}$. Let $\{u_t\}_{t\in[0,1]}$ be the $C^{1,1}$ geodesic connecting $u_0$, $u_1$. Then $[0,1]\ni t\mapsto K(u_t)$ is convex.
\end{conj}
This conjecture was verified by the fundamental work of Berman and Berndtsson \cite{Ber14-01} (c.f.  Chen-Li-Paun \cite{clp} also).
\begin{thm}
Conjecture \ref{conj2.1} is true.
\end{thm}

It turns out that the $K$-energy and also the fuctional $J_{\chi}$ can be extended to the space $(\mathcal{E}^1,d_1)$ and is convex along finite energy geodesics.
More precisely,
\begin{thm}\label{t2.4new}(\cite{Darvas1602}, Theorem 4.7)
The $K$-energy defined in (\ref{K}) can be extended to a functional $K:\mathcal{E}^1\rightarrow\mathbb{R}\cup\{+\infty\}$.
Besides, the extended functional $K|_{\mathcal{E}^1}$ is the greatest $d_1$-lower semi-continuous extension of $K|_{\mathcal{H}}$.
Moreover, $K|_{\mathcal{E}^1}$ is convex along finite energy geodesics of $\mathcal{E}_1$.
\end{thm} 
\begin{thm}
(\cite{Darvas1602}, Proposition 4.4 and 4.5)
The functional $J_{\chi}$ as defined by (\ref{J-chi}) can be extended to be a $d_1$-continuous functional on $\mathcal{E}^1$.
Besides, $J_{\chi}$ is convex along finite energy geodesics.
\end{thm}
\subsection{Calabi dream Manifolds}

Every example of a {\it Calabi dream surface} $M$ that we discusse here  is constructed from the ``outside in".  We begin with an ambient manifold that satisfies a weaker hypothesis making it easier to construct.  Then we construct $M$ as an appropriate complete intersections of ample hypersurfaces inside the ambient manifold and we encourage interested readers to Demailly-Peternell-Schneider\cite{DPS06} for further readings on this topic.

For a smooth, projective surface $M$, the ``ample cone" equals the ``big cone" if and only if the self-intersection of every irreducible curve is nonnegative. In analytic terms, the ``ample cone" equals the ``big cone" if and only if every holomorphic line bundle admitting a singular Hermitian metric of positive curvature current admits a regular Hermitian metric of positive curvature.

\begin{enumerate}

\item For every smooth, projective variety $P$ of dimension n at least 3 such that the ample cone equals the big cone, for every (n-2)-tuple of divisors $D_1, ... , D_{n-2}. \;$ If the divisor classes of $D_i$ are each globally generated, and if the $D_i$ are ``general" in their linear equivalence classes, then the surface $M = D_1 \cap ... \cap D_{n-2} $ is smooth and connected by Bertini's theorems.  If also every $ D_i$ is ample, if $K_P + (D_1 + ... + D_{n-2})$ is globally generated, and if the divisors $D_i$ are ``very general" in their linear equivalence classes, then the surface M has ample cone equal to the big cone, cf. the Noether-Lefschetz article of Ravindra and Srinivas.   Finally, if also the divisor class $K_P + (D_1 + ... + D_{n-2}) $ is ample, then $K_M$ is ample.  In that case, the smooth, projective surface $M$ has $c_1(T M)$ negative, and the self-intersection of every irreducible curve is nonnegative, and thus are Calabi dream surfaes.

\item  If $P$ and $Q$ are projective manifolds whose ample cones equal the big cones, and if there is no nonconstant morphism from the (pointed) Albanese variety of $P$ to the (pointed) Albanese variety of $Q,$ then also the product $P \times Q$ is a projective manifold whose ample cone equals the big cone.  In particular, if $P$ and $Q$ are compact Riemann surfaces of (respective) genera at least 2, and if there is no nonconstant morphism from the Jacobian of $P$ to the Jacobian of $Q$, then the product $M = P \times Q$ is a Calabi dream manifold.

\item There are many examples of smooth, projective varieties $P$ as in item 1.
When the closure of the ample cone equals the semiample cone and is finitely generated, then such a variety is precisely a ``Mori dream space" that has only one Mori chamber, yet there are examples arising from Abelian varieties where the cone is not finitely generated.  For instance, all projective varieties of Picard rank 1 trivially satisfy this property.  The next simplest class consists of all varieties that are homogeneous under the action of a complex Lie group. This class includes all Abelian varieties.  It also includes the ``projective homogeneous varieties", e.g., projective spaces, quadratic hypersurfaces in projective space, Grassmannians, (classical) flag varieties,etc. This class is also stable for products and is Calabi dream manifolds.

\item  The next simplest class consists of every projective manifold $P$ of ``cohomogeneity one", i.e., those projective manifolds that admit a holomorphic action of a complex Lie group $G$ whose orbit space is a holomorphic map from $P$ to a compact Riemann surface. These are also Calabi dream surfaces.

\end{enumerate}
Here is an interesting question about Calabi dream manifolds: how ``far" is the class of Calabi dream surfaces from the class of all  smooth minimal surfaces of general type?
%\begin{q}   There is a few open questions  for the Calabi dream manifolds:
%\begin{enumerate}
%\item  What linear combinations of Chern numbers are nonnegative for all Calabi dream manifolds?
%\item  In Corollary 1.7, can we extend it to the case of $C_1 < 0$ and without or minimal surfaces of general type? More generally,   from the point of view of symplectic topology, how far is the class of Calabi dream manifolds from the class of all  smooth minimal models of general type?
%\item  Motived by Conjecture 1 in \cite{LeSz13},  we may ask whether a K\"ahler manifold with $C_1 \leq 0$ and no subvarieties of any dimension must
%be a Calabi dream manifold?
%\end{enumerate}
%\end{q}
\section{more general cscK type equations}
First we would like to generalize our estimates on cscK to more general type of equations. More specifically, we consider the following coupled equations:
\begin{align}
\label{cscK-new1}
&\det(g_{i\bar{j}}+\varphi_{i\bar{j}})=e^F\det g_{i\bar{j}},\\
\label{cscK-new2}
&\Delta_{\varphi}F=-f+tr_{\varphi}\eta.
\end{align}
Here $f$ is a given function(not necessarily a constant) and $\eta$ is a smooth real valued closed $(1,1)$ form on $M$, written as $\eta=\sqrt{-1}\eta_{\alpha\bar{\beta}}dz_{\alpha}\wedge dz_{\bar{\beta}}$.
Observe that the equations (\ref{cscK-new1}), (\ref{cscK-new2}) combined gives
\begin{equation}
R_{\varphi}=f +tr_{\varphi}(Ric-\eta).
\end{equation}
Later on, we wish to apply our estimates to the equation (\ref{2.12}), with choice 
\begin{equation*}
f=\underline{R}-\frac{1-t}{t}\underline{\chi},\textrm{   }\eta=Ric-\frac{1-t}{t}\chi,
\end{equation*}
 and also (\ref{2.13nn}), with choice 
\begin{equation*}
f=\underline{R}-\underline{\beta}-\frac{1-t}{t}\underline{\chi},\textrm{   }\eta=Ric-\beta-\frac{1-t}{t}\chi.
\end{equation*}

The goal of this section is to prove the following apriori estimate:
\begin{thm}\label{t1.1}
Let $\varphi$ be a smooth solution to (\ref{cscK-new1}), (\ref{cscK-new2}) so that $\sup_M\varphi=0$, then there exists a 
constant $C_0>0$, depending only on the backgound metric $(M,g)$, $||f||_0$, $\max_M|\eta|_{\omega_0}$, and the upper bound of $\int_Me^FFdvol_g$ such that $||\varphi||_0\leq C_0$, and $\frac{1}{C_0}\omega_0\leq\omega_{\varphi}\leq C_0\omega_0$.
\end{thm}
The proof of this theorem is very similar to the case of cscK, and we will be suitably brief and only highlight the main differences.
Before going into the proof of Theorem \ref{t1.1}, first we notice the following corollary:
\begin{cor}\label{c1.1}
Let $\varphi$ be a smooth solution to (\ref{cscK-new1}), (\ref{cscK-new2}) normalized to be $\sup_M\varphi=0$, then for any $p<\infty$, there exist a constant $C_{0.5}$,
depending only on the background metric $(M,g)$, 
$||f||_0$, 
$\max_M|\eta|_{\omega_0}$, $p$, and the upper bound of $\int_Me^FFdvol_g$ such that $||\varphi||_{W^{4,p}}\leq C_{0.5}$, $||F||_{W^{2,p}}\leq C_{0.5}$.
\end{cor}
\begin{proof}
The proof of this corollary (assuming Theorem \ref{t1.1}) is essentially the combination of several classical elliptic estimates.
First we know from Theorem \ref{t1.1} that $\frac{1}{C_0}\omega_0\leq\omega_{\varphi}\leq C_0\omega_0$, where $C_0$ has the said dependence in this corollary.
But this means (\ref{cscK-new2}) is now uniformly elliptic with bounded right hand side.
From this we immediately know $||F||_{\alpha'}\leq C_{0.1}$, where $\alpha'$ and $C_{0.1}$ has the said dependence.
Then we go back to (\ref{cscK-new1}), we can then conclude from Evans-Krylov theorem that $||\varphi||_{2,\alpha''}\leq C_{0.2}$ for any $\alpha''<\alpha'$(see \cite{YW} for details on extension of Evans-Krylov to complex setting).
Again go back to (\ref{cscK-new2}) and notice that equation can be put in divergence form:
\begin{equation}\label{1.4n}
Re\big(\partial_i(\det(g_{\alpha\bar{\bar{\beta}}}+\varphi_{\alpha\bar{\beta}})g_{\varphi}^{i\bar{j}} F_{\bar{j}})\big)=(-f+tr_{\varphi}\eta)\det(g_{\alpha\bar{\beta}}+\varphi_{\alpha\bar{\beta}}).
\end{equation}

Here the coefficients on the left hand side is in $C^{\alpha''}$, while the right hand side is bounded.
Hence we may conclude $||F||_{1,\alpha''}\leq C_{0.3}$, from \cite{GT}, Theorem 8.32.
Then from (\ref{cscK-new1}), by differentiating both sides of the equation, we see that the first derivatives of $\varphi$ solves a linear elliptic equation with $C^{\alpha''}$ coefficient and right hand side, hence Schauder estimate applies and we conclude $\varphi\in C^{3,\alpha''}$(\cite{GT}, Theorem 6.2).
But then we go back to (\ref{1.4n}) one more time, the coefficients are in $C^{\alpha}$ for any $0<\alpha<1$ with bounded right hand side, hence we conclude $F\in C^{1,\alpha}$ for any $0<\alpha<1$.
Now the equation solved by the first derivatives of $\varphi$ will have coefficients and right hand side in $C^{\alpha}$ for any $0<\alpha<1$. Therefore $\varphi\in C^{3,\alpha}$ for any $0<\alpha<1$.

The second equation (\ref{cscK-new2}) now has $C^{1,\alpha}$ coefficient with bounded right hand side, then the classical $L^p$ estimate gives $F\in W^{2,p}$ for any finite $p$(\cite{GT}, Theorem 9.11).
Then differentiating the first equation (\ref{cscK-new1}) twice, we get a linear elliptic equation in terms of second derivatives of $\varphi$, which has $C^{\alpha}$ coefficients and $L^p$ right hand side(we already have $F\in W^{2,p}$), it follows that $\varphi\in W^{4,p}$.
\end{proof}
\begin{rem}\label{r3.2}
If we assume higher regularity of $f$ and $\eta$ on the right hand side of (\ref{cscK-new2}), it is easy to get regularity higher than $W^{4,p}$ by bootstraping.
\end{rem}
Now we can focus on proving Theorem \ref{t1.1}.

\subsection{Reduction of $C^{1,1}$ estimates to $W^{2,p}$ estimates}
This is the part where the main difference comes up with cscK case and we will highlight this difference. 
We will be brief at places where the proof works in the same way as cscK case.
 The exact result we will prove is the following:
\begin{prop}\label{p1.1}
Let $\varphi$ be a smooth solution to (\ref{cscK-new1}), (\ref{cscK-new2}), then there exists $p_n>0$, depending only on $n$, such that
\begin{equation}
\max_M|\nabla_{\varphi}F|_{\varphi}+\max_M(n+\Delta\varphi)\leq C_1.
\end{equation}
Here $C_1$ depends only on $(M,g)$, $||\varphi||_0$, $||F||_0$, $||n+\Delta\varphi||_{L^{p_n}(M)}$, $||f||_0$ and $\max_M|\eta|_{\omega_0}$.
\end{prop}
This corresponds to Theorem 4.1 in our first paper \cite{cc1}.
\begin{proof}
We can choose local coordinates so that at a point $p$ under consideration, we have
$$
g_{i\bar{j}}(p)=\delta_{ij},\,\nabla g_{i\bar{j}}(p)=0,\,\varphi_{i\bar{j}}(p)=\varphi_{i\bar{i}}(p)\delta_{ij}.
$$
Let $B:\mathbb{R}\rightarrow\mathbb{R}$ be a smooth function, we have(under above said coordinates at $p$):
\begin{equation}
\begin{split}
&e^{-B(F)}\Delta_{\varphi}(e^{B(F)}|\nabla_{\varphi}F|_{\varphi}^2)\geq2\nabla_{\varphi}F\cdot_{\varphi}\nabla_{\varphi}(\Delta_{\varphi}F)+\frac{Ric_{\varphi,i\bar{j}}F_{\bar{i}}F_j}{(1+\varphi_{i\bar{i}})(1+\varphi_{j\bar{j}})}\\
&+\frac{|F_{i\bar{j}}|^2}{(1+\varphi_{i\bar{i}})(1+\varphi_{j\bar{j}})}+B'\frac{F_iF_{j\bar{i}}F_{\bar{j}}+F_{\bar{i}}F_jF_{\bar{j}i}}{(1+\varphi_{i\bar{i}})(1+\varphi_{j\bar{j}})}+(B''|\nabla_{\varphi}F|_{\varphi}^2+B'\Delta_{\varphi}F)|\nabla_{\varphi}F|_{\varphi}^2.
\end{split}
\end{equation}
In the above, $\cdot_{\varphi}$ means the inner product is taken under the metric $\omega_{\varphi}$. 
This calculation corresponds to (4.3) in our first paper \cite{cc1} and it does not use the equation at all. Note that in the above
\begin{equation*}
Ric_{\varphi,i\bar{j}}=R_{i\bar{j}}-F_{i\bar{j}}.
\end{equation*}
As in cscK case, with the choice of $B(\lambda)=\frac{\lambda}{2}$, we obtain
\begin{equation}\label{3.7}
\begin{split}
e^{-\frac{1}{2}F}\Delta_{\varphi}(e^{\frac{1}{2}F}|\nabla_{\varphi}F|_{\varphi}^2)&\geq2\nabla_{\varphi}F\cdot_{\varphi}\nabla_{\varphi}(\Delta_{\varphi}F)+\frac{R_{i\bar{j}}F_{\bar{i}}F_j}{(1+\varphi_{i\bar{i}})(1+\varphi_{j\bar{j}})}\\
&+\frac{|F_{i\bar{j}}|^2}{(1+\varphi_{i\bar{i}})(1+\varphi_{j\bar{j}})}+\frac{1}{2}(-f+tr_{\varphi}\eta)|\nabla_{\varphi}F|_{\varphi}^2.
\end{split}
\end{equation}
In the above, we used the same crucial cancellation as in the cscK case. 
Next we can estimate
\begin{equation}
\frac{|R_{i\bar{j}}F_{\bar{i}}F_j|}{(1+\varphi_{i\bar{i}})(1+\varphi_{j\bar{j}})}\leq|Ric|_g\frac{|F_{\bar{i}}F_j|}{(1+\varphi_{i\bar{i}})(1+\varphi_{j\bar{j}})}\leq|Ric|_g|\nabla_{\varphi}F|^2_{\varphi}tr_{\varphi}g.
\end{equation}
Also
\begin{equation}
\frac{1}{2}(-f+tr_{\varphi}\eta)|\nabla_{\varphi}F|^2_{\varphi}\geq- \frac{1}{2}(||f||_0+\max_M||\eta||_gtr_{\varphi}g)|\nabla_{\varphi}F|^2_{\varphi}.
\end{equation}
Finally recall that
$$
tr_{\varphi}g\leq e^{-F}(n+\Delta\varphi)^{n-1}.
$$
Hence we obtain from (\ref{3.7}):
\begin{equation}
\begin{split}
\Delta_{\varphi}(e^{\frac{1}{2}F}|\nabla_{\varphi}F|_{\varphi}^2)&\geq 2 e^{\frac{1}{2}F}\nabla_{\varphi}F\cdot_{\varphi}\nabla_{\varphi}(\Delta_{\varphi}F)-C_{1.1}\big((n+\Delta\varphi)^{n-1}+1\big)|\nabla_{\varphi}F|_{\varphi}^2\\
&\quad\quad+\frac{1}{C_{1.1}}\frac{|F_{i\bar{\alpha}}|^2}{(1+\varphi_{i\bar{i}})(1+\varphi_{\alpha\bar{\alpha}})}.
\end{split}
\end{equation}
Here $C_{1.1}$ has the dependence stated in the proposition.
From (4.12) in our first paper, \cite{cc1}, we have
\begin{equation*}
\begin{split}
\Delta_{\varphi}(n+\Delta\varphi)&\geq-C_{1.11}(n+\Delta\varphi)^n+\Delta F-C_{1.11}\\
&\geq-C_{1.1}(n+\Delta\varphi)^n-\frac{1}{C_{1.1}}\frac{|F_{i\bar{i}}|^2}{(1+\varphi_{i\bar{i}})^2}-C_{1.1}(n+\Delta\varphi)^2-C_{1.11}\\
&\geq-C_{1.12}(n+\Delta\varphi)^n-\frac{1}{C_{1.1}}\frac{|F_{i\bar{i}}|^2}{(1+\varphi_{i\bar{i}})^2}.
\end{split}
\end{equation*}
In the last line above, we used the fact that $n+\Delta\varphi\geq ne^{\frac{F}{n}}$, which is bounded from below, and $n\geq2$.
By the same calculation as we did for cscK, if we denote $$u=e^{\frac{1}{2}F}|\nabla_{\varphi}F|_{\varphi}^2+(n+\Delta\varphi)+1,$$ we obtain
\begin{equation}
\Delta_{\varphi}u\geq 2e^{\frac{1}{2}F}\nabla_{\varphi}F\cdot_{\varphi}\nabla_{\varphi}(\Delta_{\varphi}F)-C_{1.2}(n+\Delta\varphi)^{n-1}u,
\end{equation}
Here $C_{1.2}$ has the said dependence as in proposition.
The main difference from the cscK case is that we cannot estimate the term $e^{\frac{1}{2}F}\nabla_{\varphi}F\cdot_{\varphi}\nabla_{\varphi}(\Delta_{\varphi}F)$ directly as we did for cscK, otherwise, $\nabla f$ and $\nabla \eta$ will enter into the estimates.\\

For any $p>0$, integrate the equality $$\Delta_{\varphi}(u^{2p+1})=(2p+1)2p|\nabla_{\varphi}u|_{\varphi}^2+(2p+1)u^{2p}\Delta_{\varphi}u$$ with respect to $dvol_{\varphi}$, we have
\begin{equation}\label{1.9}
\begin{split}
\int_M2&pu^{2p-1}|\nabla_{\varphi}u|_{\varphi}^2dvol_{\varphi}=\int_Mu^{2p}(-\Delta_{\varphi}u)dvol_{\varphi}\\
&\leq\int_M C_{1.2}(n+\Delta\varphi)^{n-1}u^{2p+1}dvol_{\varphi}-2\int_Me^{\frac{1}{2}F}\nabla_{\varphi}F\cdot_{\varphi}\nabla_{\varphi}(\Delta_{\varphi}F)u^{2p}dvol_{\varphi}.
\end{split}
\end{equation}
We need to integrate by parts in the last term above, then we have
\begin{equation}\label{1.10}
\begin{split}
-&\int_M2e^{\frac{1}{2}F}\nabla_{\varphi}F\cdot_{\varphi}\nabla_{\varphi}(\Delta_{\varphi}F)u^{2p}dvol_{\varphi}=\int_M4pu^{2p-1}e^{\frac{1}{2}F}\Delta_{\varphi}F\nabla_{\varphi}F\cdot_{\varphi}\nabla_{\varphi}udvol_{\varphi}\\
&+\int_M2u^{2p}e^{\frac{1}{2}F}(\Delta_{\varphi}F)^2dvol_{\varphi}+\int_Mu^{2p}e^{\frac{1}{2}F}|\nabla_{\varphi}F|_{\varphi}^2\Delta_{\varphi}Fdvol_{\varphi}.
\end{split}
\end{equation}
We wish to estimate the three terms on the right hand side of (\ref{1.10}) from above.
First,
\begin{equation}\label{1.11}
\begin{split}
\int_M4pu^{2p-1}&e^{\frac{1}{2}F}\Delta_{\varphi}F\nabla_{\varphi}F\cdot_{\varphi}\nabla_{\varphi}udvol_{\varphi}\leq \int_Mpu^{2p-1}|\nabla_{\varphi}u|_{\varphi}^2dvol_{\varphi}\\
& \qquad \qquad+ 4 \int_Mpu^{2p-1}e^F(\Delta_{\varphi}F)^2|\nabla_{\varphi}F|_{\varphi}^2dvol_{\varphi}\\
&\leq\int_Mpu^{2p-1}|\nabla_{\varphi}u|_{\varphi}^2dvol_{\varphi}+ 4 \int_Mpu^{2p}e^{\frac{1}{2}F}(\Delta_{\varphi}F)^2dvol_{\varphi}.
\end{split}
\end{equation}
Also it is clear that 
\begin{equation}\label{1.12}
\int_Mu^{2p}e^{\frac{1}{2}F}|\nabla_{\varphi}F|_{\varphi}^2\Delta_{\varphi}Fdvol_{\varphi}\leq\int_Mu^{2p+1}|\Delta_{\varphi}F|dvol_{\varphi}.
\end{equation}
Combining (\ref{1.10}), (\ref{1.11}) and (\ref{1.12}), we see
\begin{equation}
\begin{split}
-\int_Me^{\frac{1}{2}F}&\nabla_{\varphi}F\cdot_{\varphi}\nabla_{\varphi}(\Delta_{\varphi}F)u^{2p}dvol_{\varphi}\leq\int_Mpu^{2p-1}|\nabla_{\varphi}u|_{\varphi}^2dvol_{\varphi}\\
&+\int_M(4 p+2)u^{2p}e^{\frac{1}{2}F}(\Delta_{\varphi}F)^2dvol_{\varphi}+\int_Mu^{2p+1}|\Delta_{\varphi}F|dvol_{\varphi}.
\end{split}
\end{equation}
Combine with (\ref{1.9}), we obtain
\begin{equation}\label{1.14}
\begin{split}
\int_Mpu^{2p-1}|\nabla_{\varphi}u|_{\varphi}^2&dvol_{\varphi}\leq\int_MC_{1.21}(n+\Delta\varphi)^{n-1}u^{2p+1}dvol_{\varphi}\\
&+\int_Mu^{2p+1}|\Delta_{\varphi}F|dvol_{\varphi}+\int_M( 4 p+2)u^{2p}e^{\frac{1}{2}F}(\Delta_{\varphi}F)^2dvol_{\varphi}.
\end{split}
\end{equation}
In the above, we can estimate
\begin{equation}
\begin{split}
|\Delta_{\varphi}F|\leq|f|+|tr_{\varphi}\eta|\leq(||f||_0&+\max_M|\eta|_{\omega_0})(1+tr_{\varphi}g)\\
&\leq C_{1.3}(1+ne^{-F}(n+\Delta\varphi)^{n-1}).
\end{split}
\end{equation}
Recall that $n+\Delta\varphi$ is bounded from below in terms of $||F||_0$, we obtain from (\ref{1.14}) that
\begin{equation}
\int_Mpu^{2p-1}|\nabla_{\varphi}u|_{\varphi}^2dvol_g\leq \int_MC_{1.3}(p+1)(n+\Delta\varphi)^{2n-2}u^{2p+1}dvol_g.
\end{equation}
Here $C_{1.3}$ depends only on $||F||_0$, the background metric $(M,g)$, $||f||_0$, and $\max_M|\eta|_{\omega_0}$.
Above is equivalent to 
\begin{equation}\label{3.20n}
\int_M|\nabla_{\varphi}(u^{p+\frac{1}{2}})|_{\varphi}^2dvol_g\leq\frac{(p+\frac{1}{2})^2(p+1)}{p}\int_MC_{1.3}(n+\Delta\varphi)^{2n-2}u^{2p+1}dvol_g.
\end{equation}
For any $0<\eps<2$, apply H\"oler's inequality, we obtain for any $p\geq\frac{1}{2}$(by the same calculation  as in cscK case):
\begin{equation}\label{3.21n}
\bigg(\int_M|\nabla(u^{p+\frac{1}{2}})|^{2-\eps}dvol_g\bigg)^{\frac{2}{2-\eps}}\leq C_{1.4}p^2K_{\eps}\bigg(\int_Mu^{(p+\frac{1}{2})(2+\eps)}dvol_g\bigg)^{\frac{2}{2+\eps}}.
\end{equation}
Here 
$$
K_{\eps}= n^{\frac{\eps}{2-\eps}}  \bigg(\int_M(n+\Delta\varphi)^{\frac{2-\eps}{\eps}}dvol_g\bigg)^{\frac{\eps}{2-\eps}}\cdot\bigg(\int_M(n+\Delta\varphi)^{\frac{(2n-2)(2+\eps)}{\eps}}\bigg)^{\frac{\eps}{2+\eps}}.
$$
The key estimate (\ref{3.21n}) corresponds to (4.27) of our first paper, \cite{cc1}.
The passage from (\ref{3.20n}) to (\ref{3.21n}) follows the calculation from (4.22) to (4.26) of our first paper, \cite{cc1}, almost word-for-word.

After this, we choose $\eps$ sufficiently small so that
$$
\theta:=\frac{2n(2-\eps)}{2n-2+\eps}>2+\eps.
$$
Then we can apply Sobolev inequality to $u^{p+\frac{1}{2}}$ with exponent $2-\eps$ and obtain
\begin{equation}
\bigg(\int_Mu^{(p+\frac{1}{2})\theta}dvol_g\bigg)^{\frac{2}{\theta}}\leq C_{1.5}p^2\bigg(\int_Mu^{(p+\frac{1}{2})(2+\eps)}dvol_g\bigg)^{\frac{2}{2+\eps}}.
\end{equation}
This implies that for $p\geq\frac{1}{2}$, one has
\begin{equation}
||u||_{L^{(p+\frac{1}{2})\theta}}\leq(C_{1.6}p^2)^{\frac{1}{p+\frac{1}{2}}}||u||_{L^{(p+\frac{1}{2})(2+\eps)}}.
\end{equation}
Denote $\chi=\frac{\theta}{2+\eps}>1$, and choose $p+\frac{1}{2}=\chi^i$ for $i\geq0$, then from above we can conclude
$$
||u||_{L^{(2+\eps)\chi^{i+1}}}\leq \big(C_{1.6}\chi^{2i}\big)^{\chi^{-i}}||u||_{L^{(2+\eps)\chi^i}}.
$$
Iterate above estimate, and using the inequality $||u||_{L^{2+\eps}}\leq||u||_{L^1}^{\frac{1+\eps}{2+\eps}}||u||_{L^{\infty}}^{\frac{1}{1+\eps}}$ gives the desired result.
\end{proof}

\subsection{$W^{2,p}$ estimates in terms of entropy bound of $\frac{\omega_{\varphi}^n}{\omega_0^n}$}
In this subsection, we will state the estimates which ultimately shows that one can estimate the $W^{2,p}$ norm of the solution $\varphi$ in terms of the entropy bound $\int_Me^F Fdvol_g$. The proof in the cscK case carries over almost word for word.

As in the cscK case, let $\psi$ solves the following problem:
\begin{align}
\label{aux-1}
&\det(g_{i\bar{j}}+\psi_{i\bar{j}})=\frac{e^F\sqrt{F^2+1}}{\int_Me^F\sqrt{F^2+1}dvol_g} \det g,\\
\label{aux-2}
&\sup_M\psi=0.
\end{align}

\begin{thm}\label{t1.2}
Normalize $\varphi$ so that $\sup_M\varphi=0$. 
Given any $0<\eps<1$, there exists a constant $C_3$, depending only on $\eps$, the background metric $(M,g)$, the upper bound for $\int_Me^FFdvol_g$, $||f||_0$, and $\max_M|\eta|_{\omega_0}$, such that
\begin{equation}
F+\eps\psi-2(1+\max_M|\eta|_{\omega_0})\varphi\leq C_3.
\end{equation}
\end{thm}
This corresponds to Theorem 5.1 in the first paper.
\begin{proof}
Let $\eta_p:M\rightarrow\mathbb{R}_+$ be the cut-off function such that $\eta_p(p)=1$, $\eta_p\equiv1-\theta$ outside the ball $B_{\frac{d_0}{2}}(p)$, with $|\nabla\eta_p| \leq \frac{2\theta}{d_0}$ and $|\nabla^2\eta_p|\leq\frac{4\theta}{d_0^2}$.
Here $d_0$ is sufficiently small depending only on $(M,g)$ and $\theta$ is a sufficiently small constant to be chosen later.
Let $\lambda=2(1+\max_M|\eta|_{\omega_0})$, and $\delta=\frac{\alpha}{2n\lambda}$. 
Here $2\alpha$ is the $\alpha$-invariant of the K\"ahler class $(M,[\omega_0])$.
Suppose $e^{\delta(F+\eps\psi-\lambda\varphi)}$ has maximum at $p_0.\;$ Let us first calculate $\Delta_{\varphi}(e^{\delta(F+\eps\psi-\lambda\varphi)}\eta_{p_0})$.
Recall formula  (5.18) from our first paper \cite{cc1}, we have
\begin{equation}\label{3.25nn}
\begin{split}
&\Delta_{\varphi}\big(e^{\delta(F+\eps\psi-\lambda\varphi)}\eta_{p_0}\big)e^{-\delta(F+\eps\psi-\lambda\varphi)}\\
&=\eta_{p_0}\big(\delta^2|\nabla_{\varphi}(F+\eps\psi-\lambda\varphi)|_{\varphi}^2+\delta\Delta_{\varphi}(F+\eps\psi-\lambda\varphi)\big)+\Delta_{\varphi}\eta_{p_0}  \\
& \quad\quad+2\delta\nabla_{\varphi}(F+\eps\psi-\lambda\varphi)\cdot_{\varphi}\nabla_{\varphi}\eta_{p_0}.
\end{split}
\end{equation}
This does not use the equation at all.
Now we compute(similar to (5.21) in  \cite{cc1}):
\begin{equation}
\begin{split}
\Delta_{\varphi}(F+\eps\psi-\lambda\varphi)
&=-(f+\lambda n)+tr_{\varphi}(\eta+\lambda g)+\eps\Delta_{\varphi}\psi\\
&\geq(-f-\lambda n+\eps nA^{-\frac{1}{n}}(F^2+1)^{\frac{1}{2n}})+(\lambda-\eps-\max_M|\eta|_{\omega_0})tr_{\varphi}g.
\end{split}
\end{equation}
In the above $A=\int_Me^F\sqrt{F^2+1}dvol_g$, which can be bounded in terms of upper bound of $\int_Me^FFdvol_g$.
In the second line above, we used that 
$$\Delta_{\varphi}\psi\geq n(\sqrt{F^2+1}A^{-1})^{\frac{1}{n}}-tr_{\varphi}g.$$

Then we can estimate the terms in (\ref{3.25nn}) involving $\nabla_{\varphi}\eta_{p_0}$, $\Delta_{\varphi}\eta_{p_0}$  in terms of $tr_{\varphi}g$(see (5.19), (5.20) in \cite{cc1} for more details.)
Hence we conclude the following estimate(similar to (5.22) in \cite{cc1}):
\begin{equation}\label{3.27nn}
\begin{split}
e^{-\delta(F+\eps\psi-\lambda\varphi)} \Delta_{\varphi}&\big(e^{\delta(F+\eps\psi-\lambda\varphi)}\eta_{p_0}\big)\geq\delta\eta_{p_0}(-||f||_0-\lambda n+\eps nA^{-\frac{1}{n}}(F^2+1)^{\frac{1}{2n}})\\
&+\big(\delta\eta_{p_0}(\lambda-\eps-\max_M|\eta|_{\omega_0})-\frac{4\theta}{d_0^2(1-\theta)}-\frac{4\theta}{d_0^2(1-\theta)^2}\big)tr_{\varphi}g.
\end{split}
\end{equation}
Since $\eps<1$, and because of our choice of $\delta$ and $\lambda$, we know $\lambda-\eps-\max_M|\eta|_{\omega_0}>\frac{\lambda}{2}$,  and $\delta \lambda=\frac{\alpha}{2n}.$ Therefore we can choose $\theta$ small enough to make the coefficients of $tr_{\varphi}g$ in (\ref{3.27nn}) positive.
Now we drop the term involving $tr_{\varphi}g$ in (\ref{3.27nn}), and apply the Alexandrov maximum principle in $B_{d_0}(p_0)$, we obtain
\begin{equation}
\begin{split}
&\sup_{B_{d_0}(p_0)}u\eta_{p_0}\leq\sup_{\partial B_{d_0}(p_0)}u\eta_{p_0}\\
&+C_nd_0\bigg(\int_{B_{d_0}(p_0)}u^{2n}\frac{\big((-||f||_0-\lambda n+\eps nA^{-\frac{1}{n}}(F^2+1)^{\frac{1}{2n}})^-\big)^{2n}}{e^{-2F}}dvol_g\bigg).
\end{split}
\end{equation}
In the above, $u=e^{\delta(F+\eps\psi-\lambda\varphi)}$.
Then the argument proceeds the same way as in the first paper \cite{cc1}.
\end{proof}
Making use of the $\alpha$-invariant, namely the fact that $\int_Me^{-\alpha\psi}dvol_g\leq C_{3.1}$ for some $\alpha>0$ and $C_{3.1}$ depending only on background K\"ahler metric $(M,g)$, we can deduce:
\begin{cor}\label{c1.2}
Normalize $\varphi$ so that $\sup_M\varphi=0$. 
For any $1<q<\infty$, there exists a constant $C_{3.2}$, depending only on the background metric $(M,g)$, the upper bound for $\int_Me^FFdvol_g$, $||f||_0$, $\max_M|\eta|_{\omega_0}$, and $q$, such that 
\begin{equation}\label{3.25n}
\int_Me^{qF}dvol_g\leq C_{3.2}.
\end{equation}
In particular, there exists a constant $C_{3.3}$, with the same dependence as $C_{3.2}$ but not on $q$, such that
\begin{equation}
||\varphi||_0\leq C_{3.3},\,\,\,||\psi||_0\leq C_{3.3}.
\end{equation}
\end{cor}
This corresponds to Corollary 5.3 in the first paper \cite{cc1}. 
As in cscK case, the ``in particular" part follows from (\ref{3.25n}) and Kolodziej's result.

After this, one can estimate $||F||_0$.
\begin{prop}\label{p1.3}
There exist constants $C_{3.4}>0$, $C_{3.5}>0$ such that 
\begin{equation}
F\geq-C_{3.4},\,e^F\leq C_{3.5}.
\end{equation}
Here $C_{3.4}$ and $C_{3.5}$ depend on $||\varphi||_0$, $||f||_0$ and $\max_M|\eta|_{\omega_0}$
\end{prop}
This corresponds to Proposition 2.1 and Corollary 5.4 in the first paper.
Same as in cscK case, we compute $\Delta_{\varphi}(F+\lambda\varphi)$ to estimate the lower bound of $F$, and the upper bound follows from Theorem \ref{t1.2} and Corollary \ref{c1.2}.

Combining Corollary \ref{c1.2} and Proposition \ref{p1.3}, we obtain an estimate for $||\varphi||_0$ and $||F||_0$ in terms of the entropy bound $\int_Me^FFdvol_g$, $||f||_0$, $\max_M|\eta|_{\omega_0}$ and the background metric $g$ only.
As before, we have the following ``partial $C^1$ estimate".
\begin{prop}\label{p1.4}
There exists a constant $C_{3.6}$, depending only on $||\varphi||_0$, $||f||_0$, $\max_M|\eta|_{\omega_0}$, and the background metric $g$, such that
\begin{equation}
\frac{|\nabla\varphi|^2}{e^F}\leq C_{3.6}.
\end{equation}
\end{prop}
This corresponds to Theorem 2.2 in the first paper.
\begin{proof}
Let $\lambda,\,K>0$ be constants to be determined, denote $A(F,\varphi)=-(F+\lambda\varphi)+\frac{1}{2}\varphi^2$.
Then we have
\begin{equation}
\begin{split}
&\Delta_{\varphi}\big(e^{A(F,\varphi)}(|\nabla\varphi|^2+K)\big)e^{-A}\\
&=\bigg(\frac{|-F_i-\lambda\varphi_i+\varphi\varphi_i|^2}{1+\varphi_{i\bar{i}}}-\Delta_{\varphi}(F+\lambda\varphi-\frac{1}{2}\varphi^2)+\frac{|\varphi_i|^2}{1+\varphi_{i\bar{i}}}\bigg)
\\
&\times(|\nabla\varphi|^2+K)
+\Delta_{\varphi}(|\nabla\varphi|^2)+\frac{2}{1+\varphi_{i\bar{i}}}Re\big((-F_i-\lambda\varphi_i+\varphi\varphi_i)(|\nabla\varphi|^2)_{\bar{i}}\big).
\end{split}
\end{equation}
Exactly the same calculation leading to (2.21) in the first paper \cite{cc1} now gives:
\begin{equation}\label{3.34}
\begin{split}
&\Delta_{\varphi}\big(e^A(|\nabla\varphi|^2+K)\big)e^{-A}\geq|\nabla_{\varphi}(F+\lambda\varphi)-\varphi\nabla_{\varphi}\varphi|_{\varphi}^2(|\nabla\varphi|^2+K)\\
&+|\nabla_{\varphi}\varphi|_{\varphi}^2(|\nabla\varphi|^2+K)+(f-\lambda n+n\varphi+\sum_i\frac{\lambda-\eta_{i\bar{i}}-\varphi}{1+\varphi_{i\bar{i}}})(|\nabla\varphi|^2+K)\\
&+\frac{-C_{3.7}|\nabla\varphi|^2+|\varphi_{i\alpha}|^2+\varphi_{i\bar{i}}^2}{1+\varphi_{i\bar{i}}}+(-2\lambda+2\varphi)|\nabla\varphi|^2+2Re\big((F_{\alpha}+\lambda\varphi_{\alpha}-\varphi\varphi_{\alpha})\varphi_{\bar{\alpha}}\big)\\
&+\frac{2Re\big((-F_i-\lambda\varphi_i+\varphi\varphi_i)(\varphi_{\alpha}\varphi_{\bar{\alpha}\bar{i}}+\varphi_{\bar{i}}\varphi_{i\bar{i}})\big)}{1+\varphi_{i\bar{i}}}.
\end{split}
\end{equation}
In the above, $C_{3.7}$ is a constant depending only on the background metric $g$.
Following (2.22) in the first paper \cite{cc1}, we drop the complete square in (\ref{3.34}), and observe the crucial cancellation in the last two terms:
$$
(F_{\alpha}+\lambda\varphi_{\alpha}-\varphi\varphi_{\alpha})\varphi_{\bar{\alpha}}+\frac{(-F_i-\lambda\varphi_i+\varphi\varphi_i)\varphi_{\bar{i}}\varphi_{i\bar{i}}}{1+\varphi_{i\bar{i}}}=\frac{(F_i+\lambda\varphi_i-\varphi\varphi_i)\varphi_{\bar{i}}}{1+\varphi_{i\bar{i}}}.
$$
Therefore, we get following estimate similar to (2.24) in the first paper:
\begin{equation}
\begin{split}
&\Delta_{\varphi}\big(e^A(|\nabla\varphi|^2+K)\big)e^{-A}\geq K\frac{|-F_i-\lambda\varphi_i+\varphi\varphi_i|^2}{1+\varphi_{i\bar{i}}}+\frac{|\varphi_i|^2(|\nabla\varphi|^2+K)}{1+\varphi_{i\bar{i}}}\\
&+\sum_i\frac{\lambda-\eta_{i\bar{i}}-\varphi}{1+\varphi_{i\bar{i}}}(|\nabla\varphi|^2+K)+\big(-||f||_0-\lambda n+\varphi)(|\nabla\varphi|^2+K)\\
&-C_{3.7}|\nabla\varphi|^2\sum_i\frac{1}{1+\varphi_{i\bar{i}}}+\frac{\varphi_{i\bar{i}}^2}{1+\varphi_{i\bar{i}}}+(-2\lambda+2\varphi)|\nabla\varphi|^2\\
&+2Re\bigg(\frac{(F_i+\lambda\varphi_i-\varphi\varphi_i)\varphi_{\bar{i}}}{1+\varphi_{i\bar{i}}}\bigg).
\end{split}
\end{equation}
Now we choose $K=10$ and $\lambda=10(\max_M|\eta|_{\omega_0}+||\varphi||_0+C_{3.7}+1)$, then we can estimate:
\begin{equation}
\sum_i\frac{\lambda-\eta_{i\bar{i}}-\varphi}{1+\varphi_{i\bar{i}}}(|\nabla\varphi|^2+K)-C_{3.7}|\nabla\varphi|^2\sum_i\frac{1}{1+\varphi_{i\bar{i}}}\geq10|\nabla\varphi|^2\sum_i\frac{1}{1+\varphi_{i\bar{i}}}.
\end{equation}
We estimate the terms $(-||f||_0-\lambda n+ n \varphi)(|\nabla\varphi|^2+K)$, $(-2\lambda+2\varphi)|\nabla\varphi|^2$, $\frac{(F_i+\lambda\varphi_i-\varphi\varphi_i)\varphi_{\bar{i}}}{1+\varphi_{i\bar{i}}}$ and $\frac{\varphi_{i\bar{i}}^2}{1+\varphi_{i\bar{i}}}$ in the same way as we did for cscK(see (2.25), (2.26), (2.27) and (2.29) in the first paper \cite{cc1} for details).
In the end, we obtain the following estimate:
\begin{equation}
\Delta_{\varphi}\big(e^A(|\nabla\varphi|^2+K)\big)e^{-A}\geq \frac{|\varphi_i|^2|\nabla\varphi|^2}{1+\varphi_{i\bar{i}}}+9|\nabla\varphi|^2\sum_i\frac{1}{1+\varphi_{i\bar{i}}}+e^{\frac{F}{n}}-C_{3.8}(|\nabla\varphi|^2+1).
\end{equation}
Here $C_{3.8}$ is a positive constant which has the dependence described in this proposition.
This estimate corresponds to (2.30) in our first paper, \cite{cc1}.

From here on, the argument is completely the same as in cscK case.
\end{proof}
As a result of this, we deduce the following $W^{2,p}$ estimates, which were what we needed for Proposition \ref{p1.1}.
\begin{prop}\label{p1.5}
For any $p>0$, there exists a constant $\alpha(p)>0$, depending only on $p$, and another constant $C_{3.8}$, depending on $||\varphi||_0$, $||f||_0$, $\max_M|\eta|_{\omega_0}$, the background metric $(M,g)$ and $p$, such that
\begin{equation}
\int_Me^{-\alpha(p)F}(n+\Delta\varphi)^pdvol_g\leq C_{3.8}.
\end{equation}
In particular,
\begin{equation}
||n+\Delta\varphi||_{L^p(dvol_g)}\leq C_{3.9}.
\end{equation}
Here $C_{3.9}$ has the same dependence as $C_{3.8}$ but additionally on $||F||_0$.
\end{prop}
This corresponds to Theorem 3.1 in the first paper.
\begin{proof}
We start by calculating:
\begin{equation}
\begin{split}
\Delta_{\varphi}\big(&e^{-\kappa(F+\lambda\varphi)}(n+\Delta\varphi)\big)e^{\kappa(F+\lambda\varphi)}\\
&=\big(-\kappa\Delta_{\varphi}(F+\lambda\varphi)+\kappa^2|\nabla_{\varphi}(F+\lambda\varphi)|_{\varphi}^2\big)(n+\Delta\varphi)\\
&\quad\quad+\Delta_{\varphi}(n+\Delta\varphi)-2\kappa Re\bigg(\frac{(F_i+\lambda\varphi_i)(\Delta\varphi)_{\bar{i}}}{1+\varphi_{i\bar{i}}}\bigg).
\end{split}
\end{equation}
In the above, if we choose $\lambda$ so that $\lambda>2\max_M|\eta|_{\omega_0}$, then
\begin{equation}
-\kappa\Delta_{\varphi}(F+\lambda\varphi)=\kappa(R-\lambda n-tr_{\varphi}\eta+\lambda tr_{\varphi}g)\geq\kappa(-||f||_0-\lambda n)+\frac{\lambda\kappa}{2}tr_{\varphi}g.
\end{equation}
We calculate $\Delta_{\varphi}(n+\Delta\varphi)$ in exactly the same way as cscK case(see (3.4), (3.5), (3.7) in the first paper for details.)
Therefore, 
\begin{equation}
\begin{split}
\Delta_{\varphi}&\big(e^{-\kappa(F+\lambda\varphi)}(n+\Delta\varphi)\big)\geq e^{-\kappa(F+\lambda\varphi)}(\frac{\lambda\kappa}{2}-C_{3.91})(n+\Delta\varphi)\sum_i\frac{1}{1+\varphi_{i\bar{i}}}\\
&+\kappa e^{-\kappa(F+\lambda\varphi)}(-||f||_0-\lambda n)(n+\Delta\varphi)+e^{-\kappa(F+\lambda\varphi)}(\Delta F-R_g).
\end{split}
\end{equation}
In the above, $R_g$ is the scalar curvature of the background metric $\omega_0$, and $C_{3.91}$ depends only on the curvature bound of $\omega_0$.
This estimate is the analogue of (3.7) of our first paper, \cite{cc1}.
Next we use the estimate 
$$
(n+\Delta\varphi)\sum_i\frac{1}{1+\varphi_{i\bar{i}}}\geq e^{-\frac{F}{n-1}}(n+\Delta\varphi)^{1+\frac{1}{n-1}}.
$$
Denote $u=e^{-\kappa(F+\lambda\varphi)}(n+\Delta\varphi)$, as long as $\frac{\lambda\kappa}{2}-C_{3.91}>0$, we have
\begin{equation}
\begin{split}
\Delta_{\varphi}u&\geq e^{-(\kappa+\frac{1}{n-1})F-\kappa \lambda\varphi}(\frac{\lambda\kappa}{2}-C_{3.91})(n+\Delta\varphi)^{\frac{n}{n-1}}\\
&-\kappa e^{-\kappa (F+\lambda\varphi)}(\lambda n+||f||_0)(n+
\Delta\varphi)+e^{-\kappa(F+\lambda\varphi)}(\Delta F-R_g).
\end{split}
\end{equation}
For any $p\geq0$, we integrate $\Delta_{\varphi}(u^{2p+1})$ with respect to $dvol_{\varphi}=e^Fdvol_g$, we obtain
\begin{equation}\label{3.44nn}
\begin{split}
&\int_Me^{-(\kappa-\frac{n-2}{n-1})F-\kappa \lambda\varphi}\big(\frac{\lambda\kappa}{2}-C_{3.91}\big)(n+\Delta\varphi)^{\frac{n}{n-1}}u^{2p}dvol_g\\
&+\int_M2pu^{2p-1}|\nabla_{\varphi}u|_{\varphi}^2e^Fdvol_g+\int_Me^{(1-\kappa)F-\kappa \lambda\varphi}u^{2p}\Delta Fdvol_g\\
&\leq\int_M\kappa e^{(1-\kappa)F-\kappa \lambda\varphi}(\lambda n+||f||_0)(n+\Delta\varphi)u^{2p}dvol_g\\
&+\int_Me^{(1-\kappa)F-\lambda\kappa\varphi}|R_g|u^{2p}dvol_g.
\end{split}
\end{equation}
Above estimate is the analogue of (3.9) of our first paper, \cite{cc1}.
 In (\ref{3.44nn}), we need to handle the term involving $\Delta F$ via integration by parts, namely,
\begin{equation}
\begin{split}
\int_M&e^{(1-\kappa)F-\kappa\lambda\varphi}u^{2p}\Delta Fdvol_g=\int_M(\kappa-1)e^{(1-\kappa)F-\kappa\lambda\varphi}u^{2p}|\nabla F|^2dvol_g\\
&+\int_M\kappa\lambda u^{2p}\nabla\varphi\cdot\nabla Fdvol_g-\int_M2pe^{(1-\kappa)F-\kappa\lambda\varphi}u^{2p-1}\nabla u\cdot\nabla Fdvol_g.
\end{split}
\end{equation}

In order to estimate the term involving $\nabla\varphi$, we use Proposition \ref{p1.4}.
The rest of the calculation is exactly the same as cscK case.
\end{proof}

If we combine the results in Proposition \ref{p1.1}, Corollary \ref{c1.2}, Proposition \ref{p1.3}, Proposition \ref{p1.4} and Proposition \ref{p1.5}, we obtain a proof for Theorem \ref{t1.1}.

\section{$K$-energy proper implies existence of cscK}

Let the functional $I$ be as given by (\ref{IJ}), we define
$$
\mathcal{H}_0=\{\varphi\in\mathcal{H}:I(\varphi)=0\}.
$$
Following \cite{Tian97}  \cite{DR}, we introduce the following notion of properness:
\begin{defn}\label{d4.1}
We say the $K$-energy is proper with respect to $L^1$ geodesic distance
 if for any sequence $\{\varphi_i\}_{i\geq1}\subset\mathcal{H}_0$, $\lim_{i\rightarrow\infty}d_1(0,\varphi_i)=\infty$ implies $\lim_{i\rightarrow\infty}K(\varphi_i)=\infty$.
\end{defn}
The goal of this section is to prove the following existence result of cscK metrics.
\begin{thm}\label{t2.2}
Let $\beta\geq0$ be a smooth closed $(1,1)$ form. Let $K_{\beta}$ be defined as in (\ref{K-beta}). Suppose $K_{\beta}$ is proper with respect to geodesic distance $d_1$, then there exists a twisted cscK metric with respect to $\beta$(i.e, solves (\ref{2.6n})).
\end{thm}
For the converse direction, we have
\begin{thm}(main theorem of \cite{Darvas1605} and Theorem 4.13 of \cite{Darvas1602})\label{t4.2n}
Let $\beta$ be as in the previous theorem. Suppose that either
\begin{enumerate}
\item $\beta>0$;
\\
or
\item $\beta=0$ and $Aut_0(M,J)=0$.\\
Suppose there exists a twisted cscK metric with respect to $\beta$(i.e solves (\ref{2.6n})), then the functional $K_{\beta}$ is proper with respect to geodesic distance $d_1$.
\end{enumerate}
\end{thm}

In this theorem,  the case $\beta=0$ and  $Aut_0(M,J)=0$ is the main result of \cite{Darvas1605}, and the case with $\beta>0$ follows from the uniqueness of minimizers of twisted $K$-energy when the twisting form is K\"ahler (c.f. \cite{Darvas1602}, Theorem 4.13).
For completeness, we will reproduce the proof in this paper.\\

First we prove Theorem \ref{t2.2}. For this we will use the continuous path (\ref{2.13nn}) to solve (\ref{2.6n}).
Put $\chi=\omega_0$ in (\ref{2.13nn}), define
\begin{equation}\label{2.18}
S=\{t_0\in[0,1]: (\ref{2.13nn}) \textrm{ has a smooth solution for any $t\in [0,t_0]$.}\}.
\end{equation}
\begin{rem} One may also consider the set $S'$, consisting of $t_0\in[0,1]$ for which (\ref{2.13nn}) has a solution with $t=t_0$. In general, $t_0\in S'$ does not imply $[0,t_0]\subset S'$.
For instance, in \cite{chen-Zeng14}, it is shown that if a cscK metric exists (i.e,
 (\ref{2.13nn})  can be solved at $t=1.$), then we can solve this equation for all $t$ sufficiently close to $1$, for any $\beta>0$.
 However, we can always
find a $\chi >0$ such that  (\ref{2.13nn})  has no solution with $t=0.\;$.
\end{rem}
By Lemma \ref{l2.2}, we know the set $S$ is relatively open in $[0,1]$.
Also when $t=0$, (\ref{2.13nn}) has a trivial solution, namely $\varphi=0$. In particular $S\neq\emptyset$. 
The only remaining issue for the continuity method is the closedness of $S$.
Due to Theorem \ref{t1.1}, we can conclude the following criterion for closedness:
\begin{lem}\label{l2.4}
Suppose $t_i\in S$, $t_i\nearrow t_*>0$, and let $\varphi_i$ be a solution to  (\ref{2.13nn})  with $t =t_i$.
Denote $F_i=\log\frac{\omega_{\varphi_i}^n}{\omega_0^n}$.
Suppose that $\sup_i\int_Me^{F_i}F_idvol_g<\infty$, then $t_*\in S$.
\end{lem}
\begin{proof}
We just need to show (\ref{2.13nn}), or equivalently the coupled equations  (\ref{g-twisted1}), (\ref{g-twisted2}) has a smooth solution with $t=t_*$.
Indeed, the solvability of (\ref{2.13nn}) for $t<t_*$ follows from $t_i\in S$, where $t_i$ is chosen so that $t_i>t$.
Since $t_*>0$, there is no loss of generality to assume $t_i\geq\delta$ for some $\delta>0$. In light of equation (\ref{g-twisted2}), we
denote 
\[ f_i=\underline{R}-\underline{\beta}-\frac{1-t_i}{t_i}\underline{\chi},\qquad{\rm and}\qquad  \chi_i=Ric-\beta-\frac{1-t_i}{t_i}\omega_0.\]

Then we see that $(\varphi_i,F_i)$ solves (\ref{cscK-new1}), (\ref{cscK-new2}) with $f=f_i$, $\eta=\chi_i:$
\[
\Delta_{\varphi_i}F_i=-f_i+tr_{\varphi_i}\eta_i,\qquad F_i = \log {\omega_{\varphi_i}^n \over \omega_0^n}.
\]

It is clear that $\sup_i|f_i|<\infty$, $\sup_i\max_M|\chi_i|_{\omega_0}<\infty$
since $t_i\geq\delta$.
Set \[\tilde{\varphi}_i=\varphi_i-\sup_M\varphi_i,\]
then we are in a position to apply Corollary \ref{c1.1} to conclude $||\tilde{\varphi}_i||_{3,\alpha}\leq C$ for some $C>0$.
But since $f_i$ is constant, and all the higher derivatives of $\chi_i$ are also uniformly bounded independent of $i$, we see that the higher derivatives of $\tilde{\varphi}_i$ are also uniformly bounded in view of Remark \ref{r3.2}.

Hence we can take a subsequence of $\tilde{\varphi}_i$ and a smooth function $\varphi_*\in C^{\infty}(M)$ such that all derivatives of $\tilde{\varphi}_i$ converges to the corresponding derivatives for $\varphi_*$ uniformly.
Clearly $\varphi_*$ is a solution for  (\ref{2.13nn})  with $t=t_*$.
\end{proof}
To connect this criterion with properness, we need some estimates connecting the $L^1$ geodesic distance $d_1$ and the $I$ , $J_{\chi}$ functional defined in (\ref{IJ}), (\ref{J-chi}).
\begin{lem}\label{l2.5}
There exists a constant $C>0$, depending only on $n$ and the background metric $\omega_0$, such that for any $\varphi\in\mathcal{H}_0$,  we have
\begin{equation}
\begin{split}
|\sup_M\varphi|\leq C(d_1(0,\varphi)+1),\,\,|J_{\chi}(\varphi)|\leq C\max_M|\chi|_{\omega_0}d_1(0,\varphi).
\end{split}
\end{equation}
\end{lem}
\begin{proof} This is well known in the literature and we give a proof for completeness here.
We now prove the first estimate.
Let $G(x,y)$ be the Green's function defined by the metric $\omega_0$, then we can write:
\begin{equation}\label{2.21n}
\varphi(x)={1\over vol(M,\omega_0)} \int_M\varphi(y)\frac{\omega_0^n}{n!}(y)+ {1\over vol(M,\omega_0)} \int_MG(x,y)\Delta_{\omega_0}\varphi(y)\frac{\omega_0^n}{n!}(y).
\end{equation}
We know that $\sup_{M\times M}G(x,y)\leq C_{15}$, hence
\begin{equation}
\begin{split}
&\int_MG(x,y)\Delta_{\omega_0}\varphi(y)\frac{\omega_0^n}{n!}(y)=\int_M(G(x,y)-C_{15})(\Delta_{\omega_0}\varphi(y)+n)\frac{\omega_0^n}{n!}\\
&-\int_MnG(x,y)\frac{\omega_0^n}{n!}+C_{15}n\leq-n\inf_{x\in M}\int_MG(x,y)\frac{\omega_0^n}{n!}+C_{15}n:=C_{16}vol(M,\omega_0).
\end{split}
\end{equation}
Take sup in (\ref{2.21n}),
\begin{equation}
\sup_M\varphi\leq {1\over vol(M,\omega_0)} \int_M\varphi\frac{\omega_0^n}{n!}+C_{16}\leq Cd_1(0,\varphi)+C_{16}.
\end{equation}
On the other hand, since $I(\varphi)=0$, it follows from (\ref{IJ}) that $\sup_M\varphi\geq0$, so the first estimate follows.
For the second estimate, first we can calculate
\begin{equation}\label{4.6new}
\begin{split}
&\int_M\varphi\sum_{k=0}^{n-1}\chi\wedge\omega_0^k\wedge\omega_{\varphi}^{n-1-k}-n\int_M\varphi\chi\wedge\omega_0^{n-1}\\
&=\int_M\varphi\sum_{k=0}^{n-2}\chi\wedge\omega_0^k\wedge(\omega_{\varphi}^{n-1-k}-\omega_0^{n-1-
k})\\
&=\int_M-\sqrt{-1}\partial\varphi\wedge\bar{\partial}\varphi\wedge\sum_{l=0}^{n-2}(n-1-l)\chi\wedge\omega_0^{n-2-l}\wedge\omega_{\varphi}^l
\end{split}
\end{equation}
Thus,
\[
\begin{split}
& |\int_M\varphi\sum_{k=0}^{n-1}\chi\wedge\omega_0^k\wedge\omega_{\varphi}^{n-1-k}-\int_Mn\varphi\chi\wedge\omega_0^{n-1}|\\
& \leq (n-1)\max_M|\chi|_{\omega_0}\int_M-\sqrt{-1}\partial\varphi\wedge\bar{\partial}\varphi\wedge\sum_{l=0}^{n-1}\omega_0^{n-1-l}\wedge\omega_{\varphi}^l\\
&=(n-1)\max_M|\chi|_{\omega_0}\int_M\varphi(\omega_{\varphi}^n-\omega_0^n).
\end{split}
%\end{equation}
\]
%\\
%&\geq(n-1)\max_M|\chi|_{\omega_0}\int_M-\sqrt{-1}\partial\varphi\wedge\bar{\partial}\varphi\wedge\sum_{l=0}^{n-1}\omega_0^{n-1-l}\wedge\omega_{\varphi}^l\\
%&=(n-1)\max_M|\chi|_{\omega_0}\int_M\varphi(\omega_{\varphi}^n-\omega_0^n).
%\end{split}
%\end{equation}
Using  Theorem \ref{t2.3},  we conclude
$$
|\int_M\varphi\sum_{k=0}^{n-1}\chi\wedge\omega_0^k\wedge\omega_{\varphi}^{n-1-k}-\int_Mn\varphi\chi\wedge\omega_0^{n-1}|\leq C_n\max_M|\chi|_{\omega_0}d_1(0,\varphi).
$$
Similar calculation shows 
$$
|\int_M\underline{\chi}\varphi\sum_{k=0}^n\omega_0^k\wedge\omega_{\varphi}^{n-k}-(n+1)\int_M\underline{\chi}\varphi\omega_0^n|\leq C_n\max_M|\chi|_{\omega_0}d_1(0,\varphi).
$$
On the other hand, the quantities $\int_Mn\varphi\chi\wedge\omega_0^{n-1}$ and $\int_M\underline{\chi}\varphi\omega_0^n$ can be bounded in terms of $\max_M|\chi|_{\omega_0}d_1(0,\varphi)$, again due to Theorem \ref{t2.3}.
Now the claimed estimate follows from (\ref{J-chi}).
\end{proof}

From Theorem \ref{t2.2new}, any two elements in $\mathcal{E}^1$ can be connected by a ``locally finite energy geodesic" segment.
On the other hand, from Theorem 4.7 in  \cite{Darvas1602}, we know $K_{\beta}$ is convex along locally finite energy geodesic segment. 
This implies $tK_{\beta}+(1-t)J_{\omega_0}$ is convex along locally finite energy geodesics.
In view of this, we can observe:
\begin{cor}\label{c2.6}
Let $\varphi$ be a smooth solution to (\ref{g-twisted1}), (\ref{g-twisted2}) for some $t\in[0,1]$, then $\varphi$ minimizes the functional $tK_{\beta}+(1-t)J_{\omega_0}$ over $\mathcal{E}^1$.
\end{cor}
\begin{proof}
Observe that it is sufficient to show that $\varphi$ minimizes $tK_{\beta}+(1-t)J_{\omega_0}$ over $\mathcal{H}$, in view of the fact that an element in $\mathcal{E}^1$ can be approximated(under distance $d_1$) using smooth potentials with convergent entropy, as proved in Theorem 3.2, \cite{Darvas1602}, while the $J_{\chi}$ functional is continuous under $d_1$, as shown by Proposition 4.1 and Proposition 4.4 in \cite{Darvas1602}.

Next we can write $tK_{\beta}+(1-t)J_{\omega_0}=tK+J_{t\beta+(1-t)\omega_0}$. Take $\psi\in\mathcal{H}$. Let $\{u_s\}_{s\in[0,1]}$ be the $C^{1,1}$ geodesic connection $\varphi$ and $\psi$, with $u_0=\varphi$, $u_1=\psi$.
From Lemma 3.5 of \cite{Ber14-01} and the convexity of $K$-energy along $C^{1,1}$ geodesics, we conclude:
\begin{equation}\label{4.7n}
K(\psi)-K(\varphi)\geq\lim_{s\rightarrow0^+}\frac{K(u_s)-K(u_0)}{s}\geq\int_M(\underline{R}-R_{\varphi})\frac{du_s}{ds}|_{s=0}\frac{\omega_{\varphi}^n}{n!}.
\end{equation}
The first inequality used the convexity of $K$-energy along $C^{1,1}$ geodesics, proved by Berman-Berndtsson, \cite{Ber14-01}, and the second inequality is Lemma 3.5 of \cite{Ber14-01}.

On the other hand, let $\{\varphi_s\}_{s\in[0,1]}$ be any smooth curve in $\mathcal{H}$ with $\varphi_0=\varphi$, $\varphi_1=\psi$, and let $\chi\geq0$, we know from the calculation in \cite{chen00}, Proposition 2 that
\begin{equation}\label{4.8n}
\begin{split}
&J_{\chi}(\psi)-J_{\chi}(\varphi)=\int_M(tr_{\varphi}\chi-\underline{\chi})\frac{d\varphi_s}{ds}|_{s=0}\frac{\omega_{\varphi}^n}{n!}+\int_0^1(1-s)\frac{d^2}{ds^2}J_{\chi}(\varphi_s)ds\\
&=\int_M(tr_{\varphi}\chi-\underline{\chi})\frac{d\varphi_s}{ds}|_{s=0}\frac{\omega_{\varphi}^n}{n!}+\int_0^1(1-s)ds\int_M\bigg(\frac{\partial^2\varphi}{\partial s^2}-|\nabla_{\varphi_s}\frac{\partial\varphi_s}{\partial s}|^2_{\varphi_s}\bigg)tr_{\varphi_s}\chi\frac{\omega_{\varphi_s}^n}{n!}\\
&+\int_0^1(1-s)ds\int_Mg_{\varphi_s}^{i\bar{j}}g_{\varphi_s}^{k\bar{l}}\chi_{i\bar{l}}\big(\frac{\partial\varphi}{\partial s}\big)_{,k}\big(\frac{\partial\varphi}{\partial s}\big)_{,\bar{j}}\frac{\omega_{\varphi_s}^n}{n!}.
\end{split}
\end{equation}
Now we choose $\varphi_s=u_s^{\eps}$, namely the $\eps$-geodesic(which is smooth by \cite{chen991}), which means 
$$
\bigg(\frac{\partial ^2\varphi_s}{\partial s^2}-|\nabla_{\varphi_s}\frac{\partial\varphi_s}{\partial s}|^2_{\varphi_s}\bigg)\det g_{\varphi_s}=\eps\det g_0\geq0.
$$
Hence we obtain from (\ref{4.8n}) that
\begin{equation}
J_{\chi}(\psi)-J_{\chi}(\varphi)\geq\int_M(tr_{\varphi}\chi-\underline{\chi})\frac{du_s^{\eps}}{ds}|_{s=0}\frac{\omega_{\varphi}^n}{n!}.
\end{equation}
Also we know that $u_s^{\eps}\rightarrow u_s$ weakly in $W^{2,p}$ for any $p<\infty$ as $\eps\rightarrow0$.
This implies $\frac{du_s^{\eps}}{ds}|_{s=0}$, as a function on $M$, is uniformly bounded with its first derivatives.
Hence we may conclude $\frac{du_s^{\eps}}{ds}|_{s=0}\rightarrow \frac{du_s}{ds}|_{s=0}$ uniformly.
This convergence is sufficient to imply
$$
\int_M(tr_{\varphi}\chi-\underline{\chi})\frac{du^{\eps}_s}{ds}|_{s=0}\frac{\omega_{\varphi}^n}{n!}\rightarrow\int_M(tr_{\varphi}\chi-\underline{\chi})\frac{du_s}{ds}|_{s=0}\frac{\omega_{\varphi}^n}{n!},\textrm{ as $\eps\rightarrow0$.}
$$
Therefore,
\begin{equation}\label{4.10n}
J_{\chi}(\psi)-J_{\chi}(\varphi)\geq\int_M(tr_{\varphi}\chi-\underline{\chi})\frac{du_s}{ds}|_{s=0}\frac{\omega_{\varphi}^n}{n!}.
\end{equation}
Take $\chi=t\beta+(1-t)\omega_0$ in (\ref{4.10n}). Then multiply (\ref{4.7n}) by $t$, add to (\ref{4.10n}), we conclude 
\begin{equation}
K_{\beta}(\psi)-K_{\beta}(\varphi)\geq\int_M\bigg(t(\underline{R}-R_{\varphi})+(tr_{\varphi}\chi-\underline{\chi})\bigg)\frac{du_s}{ds}|_{s=0}\frac{\omega_{\varphi}^n}{n!}=0.
\end{equation}
The last equality used that $\varphi$ solves (\ref{g-twisted1}), (\ref{g-twisted2}).
\end{proof}
Using this fact, we can obtain the following improvement of Lemma \ref{l2.4}, which asserts that having control over the geodesic distance $d_1$ along the path of continuity ensures we can pass to limit.
\begin{lem}\label{l2.7}
Suppose $t_i\in S$, $t_i\nearrow t_*>0$, and let $\varphi_i$ be the solution to (\ref{2.13nn}) with $t=t_i$, normalized so that $I(\varphi_i)=0$. Suppose $\sup_id_1(0,\varphi_i)<\infty$, then $t_*\in S$.
\end{lem}
\begin{proof}
As before, we assume $t_i\geq\delta>0$.
First observe that $\sup_i(t_iK_{\beta}+(1-t_i)J_{\omega_0})(\varphi_i)<\infty$.
Indeed, we know from Corollary \ref{c2.6} that $\varphi_i$ are minimizers of $t_iK_{\beta}+(1-t_i)J_{\omega_0}$, hence
\begin{equation}\label{2.27}
\begin{split}
t_iK_{\beta}&(\varphi_i)+(1-t_i)J_{\omega_0}(\varphi_i)\leq K_{\chi,t_i}(0)=t_iK_{\beta}(0)+(1-t_i)J_{\omega_0}(0)\\
&\leq\max(K_{\beta}(0),J_{\omega_0}(0)).
\end{split}
\end{equation}

On the other hand, we know
\begin{equation}
t_iK_{\beta}(\varphi_i)+(1-t_i)J_{\omega_0}(\varphi_i)=t_i\int_Me^{F_i}F_idvol_g+t_iJ_{-Ric+\beta}(\varphi_i)+(1-t_i)J_{\omega_0}(\varphi_i).
\end{equation}
Since we assumed $\sup_id_1(0,\varphi_i)<\infty$,  Lemma \ref{l2.5} then implies that $\sup_i|J_{-Ric+\beta}(\varphi_i)|+|J_{\omega_0}(\varphi_i)|<\infty$.
 Consequently, $\sup_i\int_Me^{F_i}F_idvol_g<\infty$ since $t_i\geq\delta > 0.\;$
The result then follows from Lemma \ref{l2.4}.
\end{proof}
Now we are ready to prove Theorem \ref{t2.2}.
\begin{proof}
(of  Theorem \ref{t2.2})
Let $S$ be defined as in (\ref{2.18}), we just need to prove $S=[0,1]$.
First we know from Lemma \ref{l2.2} that $t_*>0$.
We want to show that $t_*=1$ and $1\in S$.
Indeed, if $t_*<1$, then we can take a sequence $t_i\in S$, such that $t_i\nearrow t_*$.
Let $\varphi_i$ be the solution to (\ref{2.12}) so that $I(\varphi_i)=0$.

As observed in (\ref{2.27}) above, $\sup_i\big(t_iK_{\beta}+(1-t_i)J_{\omega_0}\big)(\varphi_i)<\infty$.
On the other hand, since $0\in\mathcal{H}$ is a critical point of $J_{\omega_0}$, we know from Corollary \ref{c2.6} that $J_{\omega_0}(\varphi_i)\geq J_{\omega_0}(0)$.
Therefore we know $\sup_iK_{\beta}(\varphi_i) < \infty.\;$
By properness, we can then conclude $\sup_id_1(0,\varphi_i)<\infty$.
From Lemma \ref{l2.7} we see $t_*\in S$.
But then from Lemma \ref{l2.2} and Remark \ref{r2.3} we know $t_*+\delta'\in S$ for some $\delta'>0$ small.
This contradicts $t_*=\sup S$.
Hence we must have $t_*=1$.
Repeat the argument in this paragraph, we can finally conclude $1\in S$.
\end{proof}
For completeness, we also include here the proof of Theorem \ref{t4.2n}.
\begin{proof}
(of Theorem \ref{t4.2n})
First we assume that $\beta=0$ and $Aut_0(M,J)=0$.
Let $\varphi_0\in\mathcal{H}_0$ be such that $\omega_{\varphi_0}:=\omega_0+\sqrt{-1}\partial\bar{\partial}\varphi_0$ is cscK.
We will show that for some $\eps>0$, and for any $\psi\in\mathcal{H}_0$, $d_1(\varphi_0,\psi)\geq1$, we have $K(\psi)\geq \eps d_1(\psi,\varphi_0)+K(\varphi_0)$.

Indeed, if this were false, we will have a sequence of $\psi_i\in\mathcal{H}_0$, such that $d_1(\varphi_0,\psi_i)\geq1$, but $\eps_i:=\frac{K(\psi_i)-K(\varphi_0)}{d_1(\psi_i,\varphi_0)}\rightarrow0$. 
Let $c^i:t\in[0,d_1(\varphi_0,\psi_i)]\rightarrow \mathcal{E}^1$ be the unit speed  $C^{1,1}$ geodesic segment connecting $\varphi_0$ and $\psi_i$ \cite{chen991}.
Let $\phi_i=c^i(1)$, then $d_1(\phi_i,\varphi_0)=1$.
On the other hand, from the convexity of $K$-energy, we have
\begin{equation}
K(\phi_i)\leq \big(1-\frac{1}{d_1(\psi_i,\varphi_0)}\big)K(\varphi_0)+\frac{1}{d_1(\psi_i,\varphi_0)}K(\psi_i)=K(\varphi_0)+\eps_i.
\end{equation}
By the compactness result Lemma \ref{l2.6new}, there exists a subsequence of $\{\phi_i\}_{i\geq1}\subset\mathcal{E}^1$, denoted by $\phi_{i_j}$, such that $\phi_{i_j}\stackrel{\textrm{$d_1$}}\rightarrow\phi_{\infty}$.  
Hence $d_1(\varphi_0,\phi_{\infty})=1$.
From the lower semi-continuity of $K$-energy(Theorem 4.7 of \cite{Darvas1602}), we obtain:
\begin{equation}
K(\phi_{\infty})\leq\lim_{j\rightarrow\infty}\inf K(\phi_{i_j})\leq K(\varphi_0).
\end{equation}
But since $\varphi_0$ is a minimizer of $K$-energy over $\mathcal{E}^1$, it follows that $\phi_{\infty}$ is also a minimizer.
From Theorem 1.4 of \cite{Darvas1605}, we know $\phi_{\infty}$ is also a smooth solution to cscK equation, and there exists $g\in Aut_0(M,J)$, such that $g^*\omega_{\phi_{\infty}}=\omega_{\varphi_0}$.
But we assumed $Aut_0(M,J)=0$, hence $\omega_{\phi_{\infty}}=\omega_{\varphi_0}$.
Therefore $\phi_{\infty}-\varphi_0$ is constant.
But from the normalization $I(\phi_{\infty})=I(\varphi_0)=0$, we know $\varphi_0-\phi_{\infty}=0$, this contradicts $d_1(\varphi_0,\phi_{\infty})=1$.

Next we assume $\beta>0$. Let $\varphi^{\beta}$ solves (\ref{2.13nn}), normalized so that $I(\varphi^{\beta})=0$.
We show that for some $\eps>0$, one has $K_{\beta}(\psi)\geq\eps d_1(\varphi^{\beta},\psi)+K_{\beta}(\varphi^{\beta})$ for any $\psi\in\mathcal{H}_0$ with $d_1(\varphi^{\beta},\psi)\geq1$.

Indeed, if this were false, then there exists a sequence of $\psi_i\in\mathcal{H}_0$, such that $d_1(\varphi^{\beta},\psi_i)\geq 1$, but
$\eps_i':=\frac{K_{\beta}(\psi_i)-K_{\beta}(\varphi^{\beta})}{d_1(\psi_i,\varphi^{\beta})}\rightarrow0$. Note that $K$-energy is lower semi-continuous with respect to $d_1$ convergence and 
$J_{\beta}$ is continuous(\cite{Darvas1602}, Proposition 4.4). Hence $K_{\beta}$ is lower semicontinuous as well.
So the same argument as last paragraph applies and we get a minimizer of $K_{\beta}$, denoted as $\psi_{\infty}\in\mathcal{H}_0$, such that $d_1(\psi_{\infty},\varphi^{\beta})=1$.
But by \cite{Darvas1602}, Theorem 4.13, we know $\psi_{\infty}$ and $\varphi^{\beta}$ should differ by a constant.
Because of the normalization $I(\psi_{\infty})=I(\varphi^{\beta})=0$, we know that actually $\psi_{\infty}=\varphi^{\beta}$.
This contradicts $d_1(\psi_{\infty},\varphi^{\beta})=1$.
\end{proof}

As a corollary to this theorem, we show that the supremem of $t$ for which (\ref{2.12}) 
can be solved depends only on cohomology class of $\chi$.
More precisely,
\begin{cor}
Let $\chi_1$, $\chi_2$ be two K\"ahler forms in the same cohomology class. We define
$$
S_i=\{t_0\in[0,1]:\textrm{(\ref{2.12}) with $\chi=\chi_i$ has a smooth solution for any $t\in[0,t_0]$.}\}
$$
Then $S_1=S_2$. In particular, if we define $R([\omega_0],\chi_i)=\sup S_i$, then $R([\omega_0],\chi_1)=R([\omega_0],\chi_2)$.
\end{cor}
\begin{proof}
First we know from \cite{CoSz}, Proposition 21 and Proposition 22 that existence of smooth solutions to $tr_{\varphi}\chi_i=\underline{\chi}_i$, $i=1,\,2$ are equivalent. So we may assume both equations are solvable. Then it follows from Lemma \ref{l2.2}
that $R([\omega_0],\chi_i)>0$.
In virtue of Theorem \ref{t2.2} and Theorem \ref{t4.2n}, we just need to show for any $0<t_0\leq 1$:
\begin{equation}\label{2.37}
\textrm{$K_{\chi_1,t_0}$ is proper $\Leftrightarrow
K_{\chi_2,t_0}$ is proper.}
\end{equation}
Here $K_{\chi_i,t_0}$ is defined as in (\ref{2.10}).

Indeed, suppose $t_0\in S_1$ and $t_0<1$, then for any $0< t\leq t_0$, (\ref{2.12}) with $\chi=\chi_1$ has a solution.
From Theorem \ref{t4.2n} applied to $\beta=\frac{1-t}{t}\chi_1$, we know this implies $K_{\chi_1,t}$ is proper, for any $0<t\leq t_0$.
If (\ref{2.37}) were true, then $K_{\chi_2,t}$ is proper for any $0<t\leq t_0$.
Use Theorem \ref{t2.2} again, we know (\ref{2.12}) with $\chi=\chi_2$ is solvable for any $t\in[0,t_0]$.
This means $t_0\in S_2$.

If $t_0\in S_1$ and $t_0=1$, then it means $K$-energy is bounded from below, hence $K_{\chi_2,t}$ will be proper for $0\leq t<1$(\cite{CoSz}, Proposition 21). Then Theorem \ref{t2.2} implies (\ref{2.12}) will be solvable for $\chi=\chi_2$ and any $0\leq t<1$.
While for $t=1$, the solvability follows from the assumption that $t_0=1$, since the equation (\ref{2.12}) for $t=1$ does not involve $\chi_1$ or $\chi_2$.
Therefore $1\in S_2$.

Now we turn to the proof of (\ref{2.37}), which is an elementary calculation (c.f. \cite{Sz11}). Since $\chi_1$ and $\chi_2$ are in the same K\"ahler class, we can write
$$
\chi_1-\chi_2=\sqrt{-1}\partial\bar{\partial} \nu,\textrm{ for some smooth function $\nu$.}
$$
From (\ref{J-chi}), we can compute for $\varphi\in\mathcal{H}_0$:
\begin{equation}\label{4.17n}
\begin{split}
J_{\chi_1}(\varphi)-J_{\chi_2}(\varphi)&=\frac{1}{n!}\sum_{p=0}^{n-1}\int_M(-\varphi)\sqrt{-1}\partial\bar{\partial} \nu \wedge\omega_0^{n-p-1}\wedge\omega_{\varphi}^p\\
&=\frac{1}{n!}\sum_{p=0}^{n-1}\int_M-\nu\sqrt{-1}\partial\bar{\partial}\varphi\wedge\omega_0^{n-p-1}\wedge\omega_{\varphi}^p\\
&=\frac{-1}{n!}\int_M \nu \omega_{\varphi}^n+\int_M\frac{1}{n!} \nu \omega_0^n.
\end{split}
\end{equation}
From this it is clear that 
\begin{equation}
|J_{\chi_1}(\varphi)-J_{\chi_2}(\varphi)|\leq c_n\sup_M|\nu|.
\end{equation}
On the other hand,
\begin{equation}
|K_{\chi_1,t_0}(\varphi)-K_{\chi_2,t_0}(\varphi)|\leq (1-t_0)|J_{\chi_1}(\varphi)-J_{\chi_2}(\varphi)|\leq c_n\sup_M|\nu|.
\end{equation}
From this (\ref{2.37}) immediately follows. 
\end{proof}

\section{regularity of weak minimizers of $K$-energy}
Our main goal in this section is to show the minimizers of $K$-energy over $\mathcal{E}^1$ are always smooth. 
The main ingredients are the continuity path as well as apriori estimates obtained in section 3. The strategy of the proof is somewhat different from the usual variational problem.
Indeed, the usual strategy for variational problem will be first to take some smooth variation of the minimizer, and derive an Euler-Lagrange equation for the minimizer(in weak form).
Then one works with the Euler-Lagrange equation to obtain regularity(or partial regularity).

However, the same strategy runs into difficulty  here. Indeed, an Euler-Lagrange equation for minimizer is not apriori available, since an arbitrary smooth variation of $\varphi_*$ does not necessarily preserve the condition that $\omega_{\varphi}\geq 0$.

To get around this difficulty, we will still use the continuity path and our argument is partly inspired from \cite{Darvas1605}. The difference here is that the properness theorem (Theorem \ref{t2.2}) plays a central role. 
Here we sketch the argument. 
Take $\varphi_j$ to be smooth approximations of $\varphi_*$ (in the space $\mathcal{E}^1$), and we solve continuity path from $\varphi_j$.
That $K$-energy is bounded from below ensures the continuity path is solvable for $t<1$.
We will show the existence of a minimizer ensures that for each fixed $j$, $L^1$ geodesic distance remains bounded as $t\rightarrow1$.
Hence we can take limit as $t\rightarrow1$ and obtain a cscK potential $u_j$. Besides, such a sequence of $u_j$ will also be uniformly bounded under $L^1$ geodesic distance, which follows from the uniform boundedness of $\varphi_j$ under $L^1$ geodesic distance.
Our apriori estimates allow us to take smooth limit of $u_j$ and conclude that $u_j\rightarrow\psi$ smoothly and $\psi$ is a smooth cscK potential.
The proof is then finished once we can show $\psi$ and $\varphi_*$ only differ by an additive constant.\\
  
First we show that the existence of minimizers implies existence of smooth cscK metric.
\begin{lem}\label{l4.1}
Suppose that for some $\varphi_*\in\mathcal{E}^1$, we have $K(\varphi_*)=\inf_{\varphi\in\mathcal{E}^1}K(\varphi)$, then there exists a smooth cscK in the class $[\omega_0]$.
\end{lem}
\begin{proof}
We consider the continuity path (\ref{2.12}) with $\chi=\omega_0$. 
By assumption, $K$-energy over $\mathcal{E}^1$ is bounded from below.
Therefore the twisted $K$-energy $K_{\omega_0,t}$, defined by (\ref{2.10}) is proper for any $0\leq t<1$.
Hence we may invoke Theorem \ref{t2.2} with $\beta=\frac{1-t}{t}\omega_0$ to conclude that there exists a solution to (\ref{2.12}) for any $0<t<1$.
The only remaining issue is to see what happens in (\ref{2.12}) as $t\rightarrow1$.

Choose $t_i<1$ and $t_i\rightarrow1$, and let $\tilde{\varphi}_i$ be solutions to (\ref{2.12}) with $t=t_i$, normalized up to an additive constant so that $I(\tilde{\varphi}_i)=0$. Corollary \ref{c2.6} implies that $\tilde{\varphi}_i$ is the minimizer to $K_{\omega_0,t_i}$. Therefore we have
\begin{equation}\label{4.1}
t_iK(\varphi_*)+(1-t_i)J_{\omega_0}(\tilde{\varphi}_i)\leq t_iK(\tilde{\varphi}_i)+(1-t_i)J_{\omega_0}(\tilde{\varphi}_i)\leq t_iK(\varphi_*)+(1-t_i)J_{\omega_0}(\varphi_*).
\end{equation}
Hence (\ref{4.1}) implies that
$$
J_{\omega_0}(\tilde{\varphi}_i)\leq J_{\omega_0}(\varphi_*).
$$
On the other hand, we know $J_{\omega_0}$ is proper, in the sense that $J_{\omega_0}(\varphi)\geq \delta d_1(0,\varphi)-C$, for $\varphi\in\mathcal{H}_0$ (c.f. \cite{CoSz}, Proposition 22).
This implies that
$$
\sup_id_1(0,\tilde{\varphi}_i)\leq\frac{1}{\delta}\big(C+J_{\omega_0}(\varphi_*)\big)<\infty.
$$
Now from Lemma 4.6 we conclude that (\ref{2.12}) can be solved up to $t=1$, and we obtain the existence of a cscK potential.
\end{proof}
The main result of \cite{Darvas1605} showed the following weak-strong uniqueness property: as long as a smooth cscK exists in the K\"ahler class $[\omega_0]$, then all the minimizers of $K$-energy over $\mathcal{E}^1$ are smooth cscK. 
Therefore, we can already conclude the following result:
\begin{thm}\label{t4.1}
Let $\varphi_*\in\mathcal{E}^1$ be such that $K(\varphi_*)=\inf_{\mathcal{E}^1}K(\varphi)$. Then $\varphi_*$ is smooth, and $\omega_{\varphi_*}$ is a cscK metric.
\end{thm}

Next we will prove a more general version of Theorem \ref{t4.1}.
More precisely, we will prove:
\begin{thm}
Let $\chi\geq0$ be a closed smooth $(1,1)$ form.
Define $K_{\chi}(\varphi)=K(\varphi)+J_{\chi}(\varphi)$, where $J_{\chi}(\varphi)$ is defined by (\ref{J-chi}).
Let $\varphi_*\in\mathcal{E}^1$ be such that $K_{\chi}(\varphi_*)=\inf_{\mathcal{E}^1}K_{\chi}(\varphi)$. Then $\varphi_*$ is smooth and solves the equation $R_{\varphi}-\underline{R}=tr_{\varphi}\chi-\underline{\chi}$.
\end{thm}
Note that one can run the same argument as in 
Lemma \ref{l4.1} to show once there exists a minimizer to $K_{\chi}$, then there exists a smooth solution to 
\begin{equation}\label{4.2new}
R_{\varphi}-\underline{R}=tr_{\varphi}\chi-\underline{\chi}.
\end{equation}
However, it is not clear to us whether the argument in \cite{Darvas1605} can be adapted to this case to show a weak-strong uniqueness result. Namely if there exists a smooth solution to $R_{\varphi}-\underline{R}=tr_{\varphi}\chi-\underline{\chi}$, can one conclude all minimizers of $K_{\chi}$ are smooth?
Therefore, in the following, we will use a direct argument.
This argument is motivated from \cite{Darvas1605}, but now is more straightforward because of the use of properness theorem.

Let $\varphi_*$ be a minimizer of $K_{\chi}$.
 Then by \cite{Darvas1602}, Lemma 1.3, we may take a sequence of $\varphi_j\in\mathcal{H}$, such that $d_1(\varphi_j,\varphi_*)\rightarrow0$, and $K_{\chi}(\varphi_j)\rightarrow K_{\chi}(\varphi_*)$.
Indeed, that lemma asserts the convergence of the entropy part, but the $J_{-Ric}$ and $J_{\chi}$ are continuous under $d_1$ convergence, by \cite{Darvas1602}, Proposition 4.4.

Since there exists a minimizer to $K_{\chi}$, the functional $K_{\chi}$ is bounded from below.
On the other hand, for each fixed $j$, by \cite{CoSz}, Proposition 22, we know that $J_{\omega_{\varphi_j}}$ is proper.
Therefore, for $0\leq t<1$, the twisted $K_{\chi}$-energy $K_{\chi,\omega_{\varphi_j},t}:=tK_{\chi}+(1-t)J_{\omega_{\varphi_j}}$ is proper. 
Hence we may invoke Theorem \ref{t2.2} to conclude there exists a smooth solution to the equation
\begin{equation}\label{4.1n}
t(R_{\varphi}-\underline{R})=(1-t)(tr_{\varphi}\omega_{\varphi_j}-n)+t(tr_{\varphi}\chi-\underline{\chi}),\textrm{ for any $0\leq t<1$.}
\end{equation}
Denote the solution to be $\varphi_j^t$, normalized up to an additive constant so that $\varphi_j^t\in\mathcal{H}_0$, namely $I(\varphi_j^t)=0$.

Since $\chi\geq0$ and closed, we know that $J_{\chi}$ is convex along $C^{1,1}$ geodesic(though not necessarily strictly convex).
Hence the functional $K_{\chi}$ is convex along $C^{1,1}$ geodesic.
This again implies the convexity of $tK_{\chi}+(1-t)J_{\omega_{\varphi_j}}$ along $C^{1,1}$ geodesic.
In particular, $\varphi_j^t$ is a global minimizer of $tK_{\chi}+(1-t)J_{\omega_{\varphi_j}}$ by Corollary 4.5.

Hence we know that
\begin{equation}
tK_{\chi}(\varphi_j^t)+(1-t)J_{\omega_{\varphi_j}}(\varphi_j)\leq tK_{\chi}(\varphi_j^t)+(1-t)J_{\omega_{\varphi_j}}(\varphi_j^t)\leq tK_{\chi}(\varphi_j)+(1-t)J_{\omega_{\varphi_j}}(\varphi_j).
\end{equation}
The first inequality above uses that $\varphi_j$ minimizes $J_{\omega_{\varphi_j}}$. Hence
\begin{equation}\label{4.3n}
\sup_{0<t<1,\,j}K_{\chi}(\varphi_j^t)\leq \sup_jK_{\chi}(\varphi_j).
\end{equation}
Next we will show that the family of solution $\varphi_j^t$ are uniformly bounded in $d_1$.
First we have
\begin{equation}\label{4.5}
tK_{\chi}(\varphi_j^t)+(1-t)J_{\omega_{\varphi_j}}(\varphi_j^t)\leq tK_{\chi}(\varphi_*)+(1-t)J_{\omega_{\varphi_j}}(\varphi_*)\leq tK_{\chi}(\varphi_j^t)+(1-t)J_{\omega_{\varphi_j}}(\varphi_*).
\end{equation}
The first inequality follows from that $\varphi_j^t$ minimizes $tK_{\chi}+(1-t)J_{\omega_{\varphi_j}}$ and the second inequality follows since $\varphi_*$ minimizes $K_{\chi}$.
Therefore,
\begin{equation}\label{1.3}
J_{\omega_{\varphi_j}}(\varphi_j)\leq J_{\omega_{\varphi_j}}(\varphi_j^t)\leq J_{\omega_{\varphi_j}}(\varphi_*).
\end{equation}
The first inequality follows from that $\varphi_j$ is a minimizer of $J_{\omega_{\varphi_j}}$. 
The second inequality follows from (\ref{4.5}).
As a first observation, we have
\begin{lem}
As $j\rightarrow\infty$,
$$J_{\omega_{\varphi_j}}(\varphi_*)-J_{\omega_{\varphi_j}}(\varphi_j)\rightarrow0.$$
\end{lem}
\begin{proof}
We can compute
\begin{equation}\label{1.4}
\begin{split}
J_{\omega_{\varphi_j}}&(\varphi_*)-J_{\omega_{\varphi_j}}(\varphi_j)=\int_0^1\frac{d}{d\lambda}\big(J_{\omega_{\varphi_j}}(\lambda\varphi_*+(1-\lambda)\varphi_j)\big)d\lambda\\
&=\int_0^1d\lambda\int_M(\varphi_*-\varphi_j)\frac{\omega_{\lambda\varphi_*+(1-\lambda)\varphi_j}^{n-1}\wedge \omega_{\varphi_j}-\omega_{\lambda\varphi_*+(1-\lambda)\varphi_j}^n}{(n-1)!}\\
&=\int_0^1d\lambda\int_M\lambda(\varphi_*-\varphi_j)\wedge\sqrt{-1}\partial\bar{\partial}(\varphi_j-\varphi_*)\wedge\frac{\omega_{\lambda\varphi_*+(1-\lambda)\varphi_j}^{n-1}}{(n-1)!}\\
&=\int_0^1d\lambda\int_M\lambda\sqrt{-1}\partial(\varphi_*-\varphi_j)\wedge\bar{\partial}(\varphi_*-\varphi_j)\wedge\frac{(\lambda \omega_{\varphi_*}+(1-\lambda)\omega_{\varphi_j})^{n-1}}{(n-1)!}.
\end{split}
\end{equation}
Define
 \begin{equation}\label{1.5}
\begin{split}
I(\varphi_j,&\varphi_*)=\int_M\sqrt{-1}\partial(\varphi_j-\varphi_*)\wedge\bar{\partial}(\varphi_j-\varphi_*)\wedge\sum_{k=0}^{n-1}\omega_{\varphi_j}^k\wedge\omega_{\varphi_*}^{n-1-k}\\
&=\int_M(\varphi_j-\varphi_*)(\omega_{\varphi_*}^n-\omega_{\varphi_j}^n).
\end{split}
\end{equation}
Since we know $d_1(\varphi_j,\varphi_*)\geq\frac{1}{C} \int_M|\varphi_j-\varphi_*|(\omega_{\varphi_j}^n+\omega_{\varphi_*}^n)$ for some dimensional constant $C$, by \cite{Darvas1402}, Theorem 5.5, we have $I(\varphi_j,\varphi_*)\leq Cd_1(\varphi_j,\varphi_*)\rightarrow0$.
On the other hand, we have
$J_{\omega_{\varphi_j}}(\varphi_*)-J_{\omega_{\varphi_j}}(\varphi_j)\leq C'I(\varphi_j,\varphi_*)$ from (\ref{1.4}) and (\ref{1.5}).
Hence $J_{\omega_{\varphi_j}}(\varphi_*)-J_{\omega_{\varphi_j}}(\varphi_j)\leq C'C d_1(\varphi_j,\varphi_*)\rightarrow0$.
\end{proof}
\begin{cor}\label{c1.3}
Let $I(\varphi_j,\varphi_j^t)$ be defined similar to (\ref{1.5}), then we have $\sup_{0<t<1}I(\varphi_j,\varphi_j^t)\rightarrow 0$ as $j\rightarrow\infty$.
\end{cor}
\begin{proof}
From previous lemma and (\ref{1.3}), we know that as $j\rightarrow\infty$,
$$\sup_{0<t<1}J_{\omega_{\varphi_j}}(\varphi_j^t)-J_{\omega_{\varphi_j}}(\varphi_j)\leq J_{\omega_{\varphi_j}}(\varphi_*)-J_{\omega_{\varphi_j}}(\varphi_j)\rightarrow0.
$$
On the other hand, we know from (\ref{1.4}), (\ref{1.5}) with $\varphi_*$ replaced by $\varphi_j^t$, th following estimate holds:
$$
\frac{1}{C_n}(J_{\omega_{\varphi_j}}(\varphi_j^t)-J_{\omega_{\varphi_j}}(\varphi_j))\leq I(\varphi_j^t,\varphi_j)\leq C_n(J_{\omega_{\varphi_j}}(\varphi_j^t)-J_{\omega_{\varphi_j}}(\varphi_j)).$$ 
\end{proof}
Next we would like to show the $d_1$ distance of $\varphi_j^t$ remains uniformly bounded.
For this we will need the following key lemma:
\begin{lem}(\cite{BBEGZ}, Theorem 1.8 and Lemma 1.9)\label{l5.4n}
There exists a dimensional constant $C_n$, such that for any $u,\,v,\,w\in\mathcal{E}^1$, we have
$$
I(u,w)\leq C_n(I(u,v)+I(v,w)).
$$
Besides, we have
$$
\int_M\sqrt{-1}\partial(u-w)\wedge\bar{\partial}(u-w)\wedge\omega_v^{n-1}\leq C_nI(u,w)^{\frac{1}{2^{n-1}}}\big(I(u,v)^{1-\frac{1}{2^{n-1}}}+I(w,v)^{1-\frac{1}{2^{n-1}}}\big).
$$
\end{lem} 
As an immediate consequence of this lemma and Corollary \ref{c1.3}, we see that:
\begin{cor}\label{c1.5}
$\sup_{0<t<1}I(\varphi_j^t,\varphi_*)\rightarrow0$ as $j\rightarrow\infty$.
\end{cor}
\begin{proof}
Indeed,
$$
I(\varphi_j^t,\varphi_*)\leq C_n(I(\varphi_j^t,\varphi_j)+I(\varphi_j,\varphi_*))\leq C_n\big(I(\varphi_j^t,\varphi_j)+Cd_1(\varphi_j,\varphi_*)\big).
$$
\end{proof}

Using Lemma \ref{l5.4n}, we can show the following:
\begin{lem}
There exists a constant $C$, depending only on $\sup_jd_1(0,\varphi_j)$, $n$, such that
$$
\sup_{j,0<t<1}d_1(0,\varphi_j^t)\leq C.
$$
\end{lem}
\begin{proof}
Denote $d^c=\frac{\sqrt{-1}}{2}(\partial-\bar{\partial})$, and let $\eps>0$, we may calculate
\begin{equation}
\begin{split}
&J_{\omega_0}(\varphi_j^t)-J_{\omega_{\varphi_j}}(\varphi_j^t)\\
&=\int_0^1\int_M\frac{d}{d\lambda}(J_{\omega_0}(\lambda\varphi_j^t)-J_{\omega_{\varphi_j}}(\lambda\varphi_j^t)\big)d\lambda\\
&=\int_0^1\int_M\varphi_j^t\bigg(\frac{\omega_0\wedge\omega_{\lambda\varphi_j^t}^{n-1}}{(n-1)!}-\frac{\omega_{\varphi_j}\wedge\omega_{\lambda\varphi_j^t}^{n-1}}{(n-1)!}\bigg)d\lambda=\int_0^1\int_Md^c\varphi_j^t\wedge d\varphi_j\wedge\frac{\omega_{\lambda\varphi_j^t}^{n-1}}{(n-1)!}d\lambda\\
&\leq \eps\int_0^1\int_Md^c\varphi_j^t\wedge d\varphi_j^t\wedge\frac{\omega_{\lambda\varphi_j^t}^{n-1}}{(n-1)!}d\lambda+\frac{1}{\eps}\int_0^1\int_Md^c\varphi_j\wedge d \varphi_j\wedge\frac{\omega_{\lambda\varphi_j^t}^{n-1}}{(n-1)!}d\lambda\\
&\leq\eps C_n\int_Md^c\varphi_j^t\wedge d\varphi_j^t\wedge\sum_{k=0}^{n-1}\omega_0^k\wedge\omega_{\varphi_j^t}^{n-1-k}+\frac{C_n}{\eps}\int_Md^c\varphi_j\wedge d\varphi_j\wedge\frac{\omega_{\frac{1}{2}\varphi_j^t}^{n-1}}{(n-1)!}\\
&\leq \eps\tilde{C}_nd_1(0,\varphi_j^t)+\frac{\tilde{C}_n}{\eps}I(\varphi_j,0)^{\frac{1}{2^{n-1}}}\bigg(I(0,\frac{1}{2}\varphi_j^t)^{1-\frac{1}{2^{n-1}}}+I(\varphi_j,\frac{1}{2}\varphi_j^t)^{1-\frac{1}{2^{n-1}}}\bigg)\\
&\leq \eps\tilde{C}_nd_1(0,\varphi_j^t)+\frac{\tilde{C}_n}{\eps}I(0,\varphi_j)^{\frac{1}{2^{n-1}}}\bigg(I(0,\frac{1}{2}\varphi_j^t)^{1-\frac{1}{2^{n-1}}}\\
&\quad\quad+D_nI(0,\varphi_j)^{1-\frac{1}{2^{n-1}}}+D_nI(0,\frac{1}{2}\varphi_j^t)^{1-\frac{1}{2^{n-1}}}\bigg)\\
&\leq \eps\tilde{C}_nd_1(0,\varphi_j^t)+\eps I(0,\frac{1}{2}\varphi_j^t)+\eps^{-2^n+1}\big(\tilde{C}_n(1+D_n)\big)^{2^{n-1}}I(0,\varphi_j).
\end{split}
\end{equation}
In the first line above, we used that $J_{\omega_0}(0)=J_{\omega_{\varphi_j}}(0)=0$, which follows from (\ref{J-chi}).
We used the second inequality of Lemma \ref{l5.4n} in the passage from the 5th line to 6th line, and the first inequality in the passage from 6th line to 7th line.
In the passage from 7th line to the last line, we used Young's inequality.
Next observe that 
\begin{equation}\label{5.11n}
\begin{split}
I(0,\frac{1}{2}&\varphi_j^t)=\int_M\sqrt{-1}\partial\big(\frac{1}{2}\varphi_j^t\big)\wedge\bar{\partial}\big(\frac{1}{2}\varphi_j^t\big)\wedge\sum_{k=0}^{n-1}\omega_{\frac{1}{2}\varphi_j^t}^k\wedge \omega_0^{n-1-k}\\
&=\int_M\sqrt{-1}\partial\big(\frac{1}{2}\varphi_j^t\big)\wedge\bar{\partial}\big(\frac{1}{2}\varphi_j^t\big)\wedge\sum_{k=0}^{n-1}\frac{1}{2^k}(\omega_0+\omega_{\varphi_j^t})^k\wedge\omega_0^{n-1-k}\\
&\leq C_n\int_M\sqrt{-1}\partial\varphi_j^t\wedge\bar{\partial}\varphi_j^t\wedge\sum_{k=0}^{n-1}\omega_0^k\wedge\omega_{\varphi_j^t}^{n-1-k}=C_n\int_M\varphi_j^t(\omega_0^n-\omega_{\varphi_j^t}^n)\\
&\leq \tilde{C}_nd_1(0,\varphi_j^t).
\end{split}
\end{equation}
Hence we obtain
\begin{equation}\label{1.8}
\begin{split}
J_{\omega_0}(\varphi_j^t)&\leq J_{\omega_{\varphi_j}}(\varphi_j^t)+\eps\tilde{C}_nd_1(0,\varphi_j^t)+\eps^{-2^n+1}\big(\tilde{C}_n(1+D_n)\big)^{2^{n-1}}I(0,\varphi_j).
\end{split}
\end{equation}
On the other hand, since we know $J_{\omega_0}$ is proper in the following sense:
$$
J_{\omega_0}(\varphi)\geq\delta d_1(0,\varphi)-C,\qquad \varphi\in\mathcal{H}_0.
$$
Choose $\eps$ small enough so that 
$$
\eps\tilde{C}_n\leq\frac{\delta}{2}.
$$
Hence we obtain from (\ref{1.8}) that 
\begin{equation}
d_1(0,\varphi_j^t)\leq \frac{2}{\delta}\big(J_{\omega_{\varphi_j}}(\varphi_j^t)+\eps^{-2^n+1}\big(\tilde{C}_n(1+D_n)\big)^{2^{n-1}}I(0,\varphi_j)+C\big).
\end{equation}
Since we know that $I(0,\varphi_j)\leq Cd_1(0,\varphi_j)$, and $d_1(0,\varphi_j)$ is uniformly bounded, it only remains to find an upper bound for $J_{\omega_{\varphi_j}}(\varphi_j^t)$.
In order to bound $J_{\omega_{\varphi_j}}(\varphi_j^t)$ from above, we just need to find an upper bound for $J_{\omega_{\varphi_j}}(\varphi_*)$ thanks to (\ref{1.3}).
For this we can write:
\begin{equation}\label{5.14n}
\begin{split}
&J_{\omega_{\varphi_j}}(\varphi_*)=\int_0^1d\lambda\int_M\varphi_*\bigg(\frac{\omega_{\lambda\varphi_*}^{n-1}\wedge\omega_{\varphi_j}}{(n-1)!}-\frac{\omega_{\lambda\varphi_*}^n}{(n-1)!}\bigg)\\
&\leq \int_0^1d\lambda\int_M\varphi_*\sqrt{-1}\partial\bar{\partial}(\varphi_j-\lambda\varphi_*)\wedge\frac{\omega_{\lambda\varphi_*}^{n-1}}{(n-1)!}\\
&=\int_0^1d\lambda\int_M\lambda d^c\varphi_*\wedge d\varphi_*\wedge\frac{\omega_{\lambda\varphi_*}^{n-1}}{(n-1)!}-\int_0^1d\lambda\int_Md^c\varphi_*\wedge d\varphi_j\wedge\frac{\omega_{\lambda\varphi_*}^{n-1}}{(n-1)!}.
\end{split}
\end{equation}
In the above, $d^c=\frac{\sqrt{-1}}{2}(\partial-\bar{\partial})$, hence $d^cd=\sqrt{-1}\partial\bar{\partial}$. 
For the first term above, it can be bounded in the following way:
\begin{equation}\label{1.11}
\int_0^1d\lambda\int_M\lambda d^c\varphi_*\wedge d\varphi_*\wedge\frac{\omega_{\lambda\varphi_*}^{n-1}}{(n-1)!}\leq\int_Md^c\varphi_*\wedge d\varphi_*\wedge\sum_{k=0}^{n-1}\omega_0^k\wedge\omega_{\varphi_*}^{n-1-k}\leq Cd_1(0,\varphi_*).
\end{equation}
For the second term on the right hand side of (\ref{5.14n}),
\begin{equation}
\begin{split}
-&\int_0^1d\lambda\int_Md^c\varphi_*\wedge d\varphi_j\wedge\frac{\omega_{\lambda\varphi_*}^{n-1}}{(n-1)!}\leq\frac{1}{2}\int_0^1d\lambda\int_Md^c\varphi_*\wedge d\varphi_*\wedge\frac{\omega_{\lambda\varphi_*}^{n-1}}{(n-1)!}\\
&+\frac{1}{2}\int_0^1d\lambda\int_Md^c\varphi_j\wedge d\varphi_j\wedge\frac{\omega_{\lambda\varphi_*}^{n-1}}{(n-1)!}.
\end{split}
\end{equation}
The first term above can be estimated in the same way as in (\ref{1.11}).
For the second term above, we have
\begin{equation}
\begin{split}
&\int_0^1d\lambda\int_M\sqrt{-1}\partial\varphi_j\wedge\bar{\partial}\varphi_j\wedge\frac{\omega_{\lambda\varphi_*}^{n-1}}{(n-1)!}\\
&\leq C_n\int_M\sqrt{-1}\partial\varphi_j\wedge\bar{\partial}\varphi_j\wedge\frac{\omega_{\frac{1}{2}\varphi_*}^{n-1}}{(n-1)!}\\
&\leq C_nI(0,\varphi_j)^{\frac{1}{2^{n-1}}}\bigg(I(0,\frac{1}{2}\varphi_*)^{1-\frac{1}{2^{n-1}}}+I(\varphi_j,\frac{1}{2}\varphi_*)^{1-\frac{1}{2^{n-1}}}\bigg)\\
&\leq C_nI(0,\varphi_j)^{\frac{1}{2^{n-1}}}\bigg(I(0,\frac{1}{2}\varphi_*)^{1-\frac{1}{2^{n-1}}}+D_nI(0,\varphi_j)^{1-\frac{1}{2^{n-1}}}\\
&\quad\quad\quad\quad+D_nI(0,\frac{1}{2}\varphi_*)^{1-\frac{1}{2^{n-1}}}\bigg).
\end{split}
\end{equation}
By \cite{Darvas1402}, Theorem 5.5, $I(0,\varphi_j)$ is controlled by $d_1(0,\varphi_j)$ and the calculation in (\ref{5.11n}) shows that that $I(0,\frac{1}{2}\varphi_*)$ can be controlled in terms of $d_1(0,\varphi_*)$ respectively.
\end{proof}

Next we are ready to pass to limit.
From $\sup_{0<t<1}d_1(0,\varphi_j^t)<\infty$, we may conclude that $\sup_{j,\,0<t<1}|J_{-Ric}(\varphi_j^t)|<\infty$
 and $\sup_{j,0<t<1}|J_{\chi}(\varphi_j^t)|<\infty$ by Lemma 4.4.
 By (\ref{4.3n}) and our definition of $K_{\chi}$,  we know that $\sup_{j,t}\int_M\log\big(\frac{\omega_{\varphi_j^t}^n}{\omega_0^n}\big)\omega_{\varphi_j^t}^n<\infty$. 
Hence we may use Lemma \ref{l2.4} (the same argument works for $K_{\chi}$) to conclude that up to a subsequence of $t$, $\varphi_j^t\rightarrow u_j$ as $t\rightarrow 1$ and $u_j$ solves (\ref{4.2new}) for each $j$ with $I(u_j)=0$.
This convergence is smooth convergence due to our previous estimates.
Again due to to the last lemma, we have $\sup_jd_1(0,u_j)\leq \sup_{j,t}d_1(0,\varphi_j^t)\leq C$ for some fixed constant $C$ depending only on $n$ and $\sup_jd_1(0,\varphi_j)$.
Hence we may again assume that up to a subsequence of $j$, $u_j\rightarrow \psi$ smoothly as $j\rightarrow\infty$ and $\psi$ is a smooth solution to (\ref{4.1n}).
To finish the proof that $\varphi_*$ is smooth, we just need the following lemma:
\begin{lem}
 $\varphi_*$ and $\psi$ differ by an additive constant.
\end{lem}
\begin{proof}
By taking limit as $t\rightarrow1$, we can conclude from Corollary \ref{c1.5} that $I(u_j,\varphi_*)\rightarrow0$ as $j\rightarrow\infty$.
On the other hand, since $u_j\rightarrow\psi$ smoothly, we have $I(u_j,\psi)\rightarrow0$ as $j\rightarrow\infty$. Hence 
$$
I(\varphi_*,\psi)\leq C_n(I(u_j,\varphi_*)+I(u_j,\psi))\rightarrow0,\textrm{ as $j\rightarrow\infty$.}
$$
That is, $I(\varphi_*,\psi)=0$. 
On the other hand, from Lemma \ref{l5.8nnn}, we know $\varphi_*\in H^1(M)$ and
$$
I(\varphi_*,\psi)\geq\int_M|\nabla_{\psi}(\varphi_*-\psi)|_{\psi}^2\omega_{\psi}^n.
$$
Therefore $\psi$ and $\varphi_*$ differ only up to a constant.
\end{proof}
In the above lemma, we used the following fact.
\begin{lem}\label{l5.8nnn}
Let $\varphi\in\mathcal{E}^1$, then $\varphi\in H^1(M,\omega_0^n)$. Moreover, for any $\psi\in\mathcal{H}$, we have
\begin{equation}\label{5.18n}
I(\varphi,\psi)\geq \int_M|\nabla_{\psi}(\varphi-\psi)|_{\psi}^2\omega_{\psi}^n.
\end{equation}
In the above, $|\nabla_{\psi}(\varphi-\psi)|^2_{\psi}=g_{\psi}^{i\bar{j}}(\varphi-\psi)_i(\varphi-\psi)_{\bar{j}}$.
\end{lem}
\begin{proof}
First we assume that both $\varphi,\,\psi\in\mathcal{H}$. Then we know that
\begin{equation*}
\begin{split}
I(\varphi,\psi)&=\int_M(\varphi-\psi)(\omega_{\psi}^n-\omega_{\varphi}^n)\\
&=\int_Md^c(\varphi-\psi)\wedge d(\varphi-\psi)\wedge\sum_{k=0}^{n-1}\omega_{\varphi}^k\wedge\omega_{\psi}^{n-1-k}\\
&\geq\int_Md^c(\varphi-\psi)\wedge d(\varphi-\psi)\wedge\omega_{\psi}^{n-1}=\int_M|\nabla_{\psi}(\varphi-\psi)|_{\psi}^2\omega_{\psi}^n.
\end{split}
\end{equation*}
So (\ref{5.18n}) holds as long as $\varphi\in\mathcal{H}$.
If $\varphi\in\mathcal{E}^1$, then we can find a sequence $\phi_j\in\mathcal{H}$, such that $\phi_j$ decreases pointwisely to $\varphi$. Such approximation is possible due to the main result of \cite{BK}. 
Also due to Lemma 4.3 of \cite{Darvas1402}, we know that $d_1(\phi_j,\varphi)\rightarrow0$.
This implies that $I(\phi_j,\psi)\rightarrow I(\phi,\psi)$. Indeed, from Lemma \ref{l5.4n}, we know
\begin{equation}
|I(\phi_j,\psi)-I(\varphi,\psi)|\leq C_nI(\varphi,\phi_j)\leq \tilde{C}_nd_1(\varphi,\phi_j)\rightarrow0.
\end{equation}

Since (\ref{5.18n}) holds with $\varphi$ replaced by $\varphi_j$, we see that
\begin{equation}\label{5.20n}
\int_M|\nabla_{\psi}(\phi_j-\psi)|^2_{\psi}\omega_{\psi}^n\leq I(\phi_j,\psi)\rightarrow I(\varphi,\psi).
\end{equation}
From $\sup_jd_1(0,\phi_j)<\infty$, we know that $\sup_j\int_M|\phi_j|dvol_g<\infty$. Now (\ref{5.20n}) shows $\phi_j$ is uniformly bounded in $H^1(M,\omega_{\psi}^n)$.
Hence we can find a subsequence of $\phi_j$
which converges weakly in $H^1(M,\omega_{\psi}^n)$, strongly in $L^2(M,\omega_{\psi}^n)$. Clearly this limit must be $\varphi$. This shows $\varphi\in H^1(M,\omega_{\psi}^n)$, hence also in $H^1(M,\omega_0^n)$. Also we can conclude from (\ref{5.20n}) that
$$
\int_M|\nabla_{\psi}(\varphi-\psi)|^2_{\psi}\omega_{\psi}^n\leq\lim\inf_{j\rightarrow\infty}\int_M|\nabla_{\psi}(\phi_j-\psi)|^2\omega_{\psi}^n\leq\lim\inf_jI(\phi_j,\psi)=I(\varphi,\psi).
$$
\end{proof}

\section{Existence of cscK and geodesic stability}
In this section, we prove Theorem \ref{t1.1new}.
Similar to the definition of $\mathcal{H}_0$, we define
$$
\mathcal{E}^1_0=\mathcal{E}^1\cap\{u:I(u)=0\}.
$$
Here $I(u)$ for $u\in\mathcal{E}^1$ is understood as the continuous extension of the functional $I$ from $\mathcal{H}$ to $\mathcal{E}^1$.
This is possible because of Proposition 4.1 in \cite{Darvas1602}.
Also we notice that for any $u_0$, $u_1\in\mathcal{E}_0^1$, the finite energy geodesic segment (defined by Theorem \ref{t2.2new}) $[0,1]\ni t\rightarrow \mathcal{E}^1$ will actually lie in $\mathcal{E}_0^1.\;$
This follows from the fact that the $I$ functional is affine on $C^{1,1}$ geodesics and $I$ can be continuously extended to the space $\mathcal E^1$.
As before, $\beta\geq0$ is a smooth closed $(1,1)$ form. We will first prove the following result in this section, which covers Theorem \ref{t1.1new}.
\begin{thm}\label{t3.1}
Suppose that either
\begin{enumerate}
\item $\beta>0$ everywhere;\\
or
\item $\beta=0$ everywhere and $Aut_0(M,J)=0$.
\end{enumerate}
Then the following statements are equivalent:
\begin{enumerate}
\item There exists no twisted cscK metric with respect to $\beta$ in $\mathcal{H}_0$.
\item There is an infinite geodesic ray $\rho_t$ with locally finite energy, $t\in[0,\infty)$ in $\mathcal{E}_0^1$, such that the functional $K_{\beta}$ is non-increasing along the ray.
\item For any $\phi\in\mathcal{E}_0^1$ with $K(\phi)<\infty$, there is a locally finite energy geodesic ray starting at $\phi$, such that the functional $K_{\beta}$ is non-increasing along the ray.
\end{enumerate}
In the case $\beta>0$, then from (1) one can additionally conclude $K_{\beta}$ is strictly decreasing in (2) and (3) above.
\end{thm}
\begin{defn}
Let $[0,\infty)\ni t\rightarrow u_t\in\mathcal{E}^1$ be a continuous curve. Then we say $u_t$ is an infinite geodesic ray with locally finite energy, if the following hold:
\begin{enumerate}
\item $d_1(u_t,u_s)=c|t-s|$ for some constant $c>0$ and any $s,\,t\in[0,\infty)$.
\item For any $K>0$, $[0,K]\ni t\rightarrow u_t$ is a finite energy geodesic segment in the sense defined by Theorem \ref{t2.2new}.
\end{enumerate}
\end{defn}

\begin{rem}
Observe that the implication $(3)\Rightarrow (2)$ is trivial.
$(2)\Rightarrow(1)$ follows from Theorem \ref{t4.2n}.
We will use our apriori estimates and the continuity path (\ref{2.12}) to resolve the implication $(1)\Rightarrow(3)$.
 We are partly motivated from arguments in the proof of Theorem 6.5 of \cite{Darvas1602}.
\end{rem}

Next we observe the following lemma:
\begin{lem}
Consider the continuity path (\ref{2.13nn}).
Suppose there is no twisted cscK metric with respect to $\beta$ in 
K\"ahler class $[\omega_0]$.
Denote $t_*=\sup S$, where the set $S$ is defined in (\ref{2.18}). 
Let $S\ni t_i\nearrow t_*$. Denote $\varphi_i$ to be the solution to (\ref{2.12}) with $t=t_i$, normalized so that $I(\varphi_i)=0$.
Then we have $\sup_id_1(0,\varphi_i)=\infty$.
\end{lem}
\begin{proof}

Suppose otherwise, then $\sup_id_1(0,\varphi_i)<\infty$.
We can apply Lemma \ref{l2.7} to conclude $t_*\in S$.
If $t_*<1$, then we conclude from Lemma \ref{l2.2} that $t_*+\delta'\in S$ for some $\delta'>0$ sufficiently small.
This contradicts $t_*=\sup S$.
If $t_*=1$, then $1\in S$. But this will contradict our assumption that there is no cscK metric in $[\omega_0]$.
In either case, the contradiction shows one cannot have $\sup_id_1(0,\varphi_i)<\infty$.
\end{proof}
With the help of above lemma, we are ready to prove $(1)\Rightarrow(3)$ in Theorem \ref{t3.1}.
\begin{proof}
\sloppy Consider the continuity path (\ref{2.13nn}) as in Lemma 6.3,  we know that $\sup_id_1(0,\varphi_i)=\infty$.
Hence we may take a subsequence $\varphi_{i_j}$, such that $d_1(0,\varphi_{i_j})\nearrow\infty$.
We will construct a geodesic ray as described in Theorem \ref{t3.1}, point (2) out of this subsequence $\varphi_{i_j}$.
For simplicity, we will still denote this subsequence by $\varphi_i$.

By Theorem \ref{t2.2new}, there exists a unit speed finite energy $d_1$-geodesic segment connectiong $\phi$ and $\varphi_i$, such that the functional $I$ is affine on the segment.
Indeed, one can check $I$ is affine on $C^{1,1}$ geodesic and the extension to $d_1$-geodesic follows from continuity of the functional $I$(c.f \cite{Darvas1602}, Proposition 4.1).

Denote this geodesic by $c^i:[0,d_1(\phi,\varphi_i)]\rightarrow\mathcal{E}^1$.
Since $I(\phi)=I(\varphi_i)=0$, we know $I=0$ on $c^i$. In other words,  $c^i:[0,d_1(\phi,\varphi_i)]\rightarrow\mathcal{E}_0^1$.
As noted in (\ref{2.27}), we have $$\displaystyle \sup_i\big(t_iK_{\beta}+(1-t_i)J_{\omega_0}\big)(\varphi_i)\leq \max(K_{\beta}(0),J_{\omega_0}(0)).$$
On the other hand, since the functional $J_{\omega_0}$ is convex along $C^{1,1}$ geodesic, and we know $0$ is a critical point of $J_{\omega_0}$, we see that
\begin{equation}
J_{\omega_0}(\varphi_i)\geq J_{\omega_0}(0).
\end{equation}
Therefore
\begin{equation}
K_{\beta}(\varphi_i)\leq\frac{\max(K_{\beta}(0),J_{\omega_0}(0))-(1-t_i)J_{\omega_0}(0)}{t_i}\leq C.
\end{equation}
Hence from the convexity of $K_{\beta}$-energy as remarked before, we obtain for any $l\in[0,d_1(\phi,\varphi_i)]$,
\begin{equation}\label{3.6}
K_{\beta}(c^i(l))\leq(1-\frac{l}{d_1(\phi,\varphi_i)})K_{\beta}(\phi)+\frac{l}{d_1(\phi,\varphi_i)}K_{\beta}(\varphi_i)\leq\max(K_{\beta}(\phi),C).
\end{equation}

Therefore, for each fixed $l$. if we consider the sequence $\{c^i(l)\}_{d_1(\phi,\varphi_i)\geq l}\subset\mathcal{E}^1$, it satisfies the assumption in Lemma \ref{l2.6new}. Indeed, $d_1(\phi,c^i(l))=l,\,\forall i$, which implies $\sup_i|J_{\beta}(c^i(l))|$ uniformly bounded for fixed $l$(by Lemma \ref{l2.5}).
Therefore, we have $K$-energy is uniformly bounded and we may apply Lemma \ref{l2.6new}.

Hence we may take a subsequence $c^{i_j}(l)$, such that $c^{i_j}(l)\rightarrow c^{\infty}(l)$ for some element $c^{\infty}(l)\in\mathcal{E}^1$ as $j\rightarrow\infty$.
Since the functional $I$ is continuous under $d_1$ convergence, we obtain $c^{\infty}(l)\in \mathcal{E}^1_0$ as well.
Clearly we may apply this argument to each $l\in\mathbb{Q}$, then by Cantor's diagnal sequence argument, we can take a subsequence of $\varphi_i$, denoted by $\varphi_{i_j}$, such that 
\begin{equation}\label{5.4new}
c^{i_j}(l)\rightarrow c^{\infty}(l) \textrm{ in $d_1$},\textrm{ as $j\rightarrow\infty$, for any $l\in\mathbb{Q}$.}
\end{equation}
Since $c^{i_j}$ are unit speed geodesic segment, we see that for any $r,\,s\in\mathbb{Q}$, 
with $0\leq r,\,s\leq d_1(\phi,\varphi_{i_j})$, 
we have $d_1(c^{i_j}(r),c^{i_j}(s))=|r-s|$.
Sending $j\rightarrow\infty$ gives 
\begin{equation}\label{3.9}
d_1(c^{\infty}(r),c^{\infty}(s))=|r-s|,\textrm{ for any $0\leq r,\,s\in\mathbb{Q}$.}
\end{equation}
We can then define $c^{\infty}(r)$ for all $r\in\mathbb{R}$ by requiring $c^{\infty}(r)=d_1-\lim_{r_k\in\mathbb{Q},r_k\rightarrow r} c^{\infty}(r_k)$.
From property (\ref{3.9}) it is easy to see this is well defined, i.e, the said limit exists and does not depend on our choice of sequence $r_k$.
Hence $[0,\infty)\ni r\rightarrow c^{\infty}(r)$ is a unit speed geodesic ray in $\mathcal{E}^1_0$.
Besides, if we apply Proposition \ref{p2.4new} to $[0,r_k]$ for any $r_k>0$, $r_k\in\mathbb{Q}$, we know $c^{i_j}(r)\rightarrow u_k(r)$ for any $r\in[0, r_k]$.
Here $[0,r_k]\ni r\rightarrow u_k(r)$ is the finite energy geodesic segment connecting $\phi$ and $c^{\infty}(r_k)$.
Hence we know $c^{\infty}(r)=u_k(r)$ for any $r\in[0,r_k]\cap\mathbb{Q}$, by (\ref{5.4new}). Therefore $c^{\infty}(r)=u_k(r)$ for any $r\in[0,r_k]$ by density.
Therefore, we have shown $c^{\infty}|_{[0,d_1(\phi,c^{\infty}(r))]}$ is the finite energy geodesic segment connecting $\phi$ and $c^{\infty}(r)$ for $r\in\mathbb{Q}$.
It is easy to extend this to all $r\in\mathbb{R}_+$ by rescaling in time and apply Proposition \ref{p2.4new} again.

We can now invoke Theorem 4.7, Proposition 4.5 of \cite{Darvas1602} to conclude $r\longmapsto K(c^{\infty}(r))$, $r\longmapsto J_{\beta}(c^{\infty}(r))$ is convex. Hence $r\longmapsto K_{\beta}(c^{\infty}(r))$ is convex as well.

Now from the lower semi-continuity of $K_{\beta}$-energy under $d_1$-convergence, we obtain from (\ref{3.6}) that
\begin{equation}
K_{\beta}(c^{\infty}(r))\leq\lim\inf_{j\rightarrow\infty}K_{\beta}(c^{i_j}(r))\leq\max(K_{\beta}(\phi),C),\textrm{ for all $r\in\mathbb{Q}$.}
\end{equation}
Use the lower semi-continuity again, we deduce
\begin{equation}
K_{\beta}(c^{\infty}(r))\leq\lim\inf_{k\rightarrow\infty}K_{\beta}(c^{\infty}(r_k))\leq\max(K_{\beta}(\phi),C).
\end{equation}
Therefore, $(0,\infty)\ni r\longmapsto K_{\beta}(c^{\infty}(r))$ is both convex and bounded, this forces $K_{\beta}$-energy must be decreasing along $c^{\infty}$.

To see the ``in addition" part, if $K_{\beta}$ is not strictly decreasing, them from the convexity of $r\longmapsto K_{\beta}(c^{\infty}(r))$, we can conclude that for some $r_0>0$, $K_{\beta}(c^{\infty}(r))$ remains a constant for $r\geq r_0$.
Since both $K$ and $J_{\beta}$ are convex, we know $J_{\beta}$ remains linear for $r\geq r_0$.
Now \cite{Darvas1602}, Theorem 4.12 shows $c^{\infty}(r_1)=c^{\infty}(r_r)+const$ for any $r_1,\,r_2\geq r_0$.
Because of the normalization $I(c^{\infty}(r))=0$, we know $c^{\infty}(r_1)=c^{\infty}(r_2)$ for any $r_1,\,r_2\geq r_0$.
But this contradicts $d_1(c^{\infty}(r_1),c^{\infty}(r_2))=|r_1-r_2|$ for any $r_1,\,r_2\geq0$.
\end{proof}
Finally, the implication $(2)\Rightarrow(1)$ follows immediately from Theorem \ref{t4.2n}.
\begin{proof}
Suppose otherwise, namely there exists a twisted cscK metric with respect to $\beta$ in $\mathcal{H}_0$, denoted by $\varphi^{\beta}$.
Then we can conclude from Theorem \ref{t4.2n} that the twisted $K$-energy $K_{\beta}$ is proper. In particular, $K_{\beta}\rightarrow+\infty$ along any locally finite energy geodesic ray. This contradicts the assumption in (2).
\end{proof}
We can deduce the following immediate consequence of Theorem \ref{t3.1}.
\begin{cor}
Let $0<t_0<1$, and let $\chi$ be a K\"ahler form.
Then the following statements are equivalent:
\begin{enumerate}
\item There is no twisted cscK metric with $t=t_0$ in $\mathcal{H}_0$(i.e solves (\ref{2.12}) with $t=t_0$).
\item There is an infinite geodesic ray $\rho_t$ of locally finite energy,
 $t\in[0,\infty)$ in $\mathcal{E}^1_0$, such that the twisted $K$-energy $K_{\chi,t_0}$(defined by (\ref{2.10})) 
is strictly decreasing along the ray.
\item For any $\phi\in \mathcal{E}_0^1$ with $K(\phi)<\infty$, there is a locally finite energy  geodesic ray  starting at $\phi$, such that the twisted $K$-energy $K_{\chi,t_0}$(defined by (\ref{2.10})) 
is strictly decreasing along the ray.
\end{enumerate}
\end{cor}

Also we can show Theorem \ref{t1.2n} as a consequence.
\begin{proof}
(of Theorem \ref{t1.2n}) First we prove the necessary part.  Assume $(M,[\omega_0])$ admits a cscK metric. 
Denote $\varphi_0$ be the corresponding cscK potential.
Recall we have shown in the proof of Theorem \ref{t4.2n}(the direction existence implies properness) that for all $\psi\in\mathcal{E}^1_0$, with $d_1(\psi,\varphi_0)\geq1$, one has $K(\psi)\geq\eps d_1(\psi,\varphi_0)+K(\varphi_0)$.
Let $\phi\in\mathcal{E}_0^1$ and $\rho:[0,\infty)\ni t\mapsto\mathcal{E}_0^1$ be a locally finite energy geodesic ray initiating from $\phi$. We can assume $\rho(t)$ has unit speed.
Then as long as $d_1(\rho(t),\varphi_0)\geq1$, one has
\begin{equation}
\begin{split}
\frac{K(\rho(t))-K(\phi)}{t}&\geq\frac{\eps d_1(\rho(t),\varphi_0)+K(\varphi_0)-K(\phi)}{t}\\
&\geq\frac{\eps d_1(\rho(t),\phi)-\eps d_1(\phi,\varphi_0)+K(\varphi_0)-K(\phi)}{t}\\
&=\eps-\frac{\eps d_1(\phi,\varphi_0)-K(\varphi_0)+K(\phi)}{t}.
\end{split}
\end{equation}
This implies 
$$
\lim_{t\rightarrow\infty}\inf\frac{K(\rho(t))-K(\phi)}{t}\geq\eps.
$$
In particular this means $\yen([\rho])\geq\eps$. Thus, $(M,[\omega_0])$ is geodesic stable.\\

Now we want to show the converse. We assume $(M,[\omega_0])$ is geodesic stable and we want to prove that there is a cscK metric in the K\"ahler 
class. Suppose otherwise, then according to Theorem \ref{t3.1} with $\beta=0$, point (3), we know that there exists a locally finite energy geodesic ray $\rho:[0, \infty)\ni t\mapsto\mathcal{E}_0^1$, initiating from $\phi\in\mathcal{E}_0^1$ with $K(\phi)<\infty$, such that the $K$-energy is non-increasing.
It is clear that for this geodesic ray, one has $\yen([\rho])\leq0$.
This contradicts the assumption of geodesic stability at $\varphi$. This finishes the proof.
\end{proof}

\noindent Xiuxiong Chen\\
University of Science and Technology of China and Stony Brook University\\

\noindent Jingrui Cheng\\
University of Wisconsin at Madison.

\end{document}